\documentclass[11pt,centertags,reqno,twoside]{amsart}
\usepackage{amssymb}
\usepackage{mathptmx}
\usepackage{color}

\DeclareSymbolFont{SY}{U}{psy}{m}{n}
\DeclareMathSymbol{\emptyset}{\mathord}{SY}{'306}

\renewcommand{\eqref}[1]{{\rm(\ref{#1})}}

\newcommand{\bbC}{{\mathbb C}}
\newcommand{\bbR}{{\mathbb R}}
\newcommand{\bbZ}{{\mathbb Z}}
\newcommand{\bbN}{{\mathbb N}}

\newcommand{\cB}{{\mathcal B}}

\newcommand{\cG}{{\mathcal G}}

\newcommand{\conv}{\mathop{\rm conv}}

\newcommand{\cP}{{\mathcal P}}

\newcommand{\cT}{{\mathcal T}}

\newcommand{\sE}{{\sf E}}

\newcommand{\fG}{\mathfrak{G}}
\newcommand{\fH}{\mathfrak{H}}
\newcommand{\fK}{\mathfrak{K}}
\newcommand{\fL}{\mathfrak{L}}
\newcommand{\fM}{\mathfrak{M}}

\newcommand{\fT}{\mathfrak{T}}

\newcommand{\diag}{\mathop{\rm diag}}
\newcommand{\dist}{\mathop{\rm dist}}

\newcommand{\Real}{\mathop{\rm Re}}

\newcommand{\be}{\begin{equation}}
\newcommand{\ee}{\end{equation}}

 \DeclareMathOperator{\spec}{spec}

\newcommand{\dom}{\mathop{\mathrm{Dom}}}
\newcommand{\Dom}{\mathop{\mathrm{Dom}}}
\newcommand{\Ker}{\mathop{\mathrm{Ker}}}
\newcommand{\ran}{\mathop{\mathrm{Ran}}}
\newcommand{\Ran}{\mathop{\mathrm{Ran}}}

\allowdisplaybreaks

\numberwithin{equation}{section}

\newtheorem{theorem}{Theorem}[section]
\newtheorem{corollary}[theorem]{Corollary}
\newtheorem{lemma}[theorem]{Lemma}

\newtheorem{hypothesis}[theorem]{Assumption}
\theoremstyle{definition}
\newtheorem{definition}[theorem]{Definition}
\newtheorem{remark}[theorem]{Remark}
\newtheorem{example}[theorem]{Example}

\DeclareMathOperator{\arctanh}{arctanh}
\newcommand{\efrac}[2]{\genfrac{}{}{0pt}{}{#1}{#2}}
\newcommand{\la}{\lambda}
\renewcommand{\Re}{{\rm Re\,}}
\renewcommand{\Im}{{\rm Im\,}}

\begin{document}

\title[Bounds on the spectrum and reducing subspaces]
{Bounds on the spectrum and reducing subspaces \\ of a
$J$-self-adjoint operator}

\author[S. Albeverio, A. K. Motovilov, and C. Tretter]
{Sergio Albeverio,  Alexander K. Motovilov, and Christiane Tretter}

\address{Sergio Albeverio,
Institut f\"ur Angewandte Mathematik, Universit\"at Bonn, Endenicher
Allee 60, D-53115 Bonn, Germany; SFB 611 and HCM, Bonn; BiBoS,
Bielefeld-Bonn; CERFIM, Lo\-car\-no; Accademia di Architettura, USI,
Mendrisio} \email{albeverio@uni-bonn.de}

\address{Alexander K. Motovilov, Bogoliubov Laboratory of
Theoretical Physics, JINR, Joliot-Cu\-rie 6, 141980 Dubna, Moscow
Region, Russia} \email{motovilv@theor.jinr.ru}

\address{Christiane Tretter, Mathematisches Institut,
Universit\"at Bern, Sidlerstrasse 5, CH-3012 Bern, Switzerland}
\email{tretter@math.unibe.ch}

\date{September 7, 2009}



\begin{abstract}
Given a self-adjoint involution $J$ on a Hilbert space $\fH$, we
consider a $J$-self-adjoint operator $L=A+V$ on $\fH$ where $A$ is a
possibly unbounded self-adjoint operator commuting with $J$ and $V$
a bounded $J$-self-adjoint operator anti-commuting with $J$. We
establish optimal estimates on the position of the spectrum of $L$
with respect to the spectrum of $A$ and we obtain norm bounds on the
operator angles between maximal uniformly definite reducing
subspaces of the unperturbed operator $A$ and the perturbed operator
$L$. All the bounds are given in terms of the norm of $V$ and the
distances between pairs of disjoint spectral sets associated with
the operator $L$ and/or the operator $A$. As an example, the quantum
harmonic oscillator under a $\cP\cT$-symmetric perturbation is
discussed. The sharp norm bounds obtained for the operator angles
generalize the celebrated Davis-Kahan trigonometric theorems to the
case of $J$-self-adjoint perturbations.
\end{abstract}

\maketitle

\section{Introduction}
\label{SIntro}
Let $\fH$ be a Hilbert space and $J$ a self-adjoint involution on
$\fH$, that is, $J^*=J$ and $J^2=I$, where $J\neq I$ and $I$ denotes
the identity operator. A linear operator $L$ on $\fH$ is called
$J$-self-adjoint if the product $JL$ is a self-adjoint operator on
$\fH$, that is, $(JL)^*=JL$.

In this paper we consider a $J$-self-adjoint operator $L$ of the
form $L=A+V$ where $A$ is a (possibly unbounded) self-adjoint
operator on $\fH$ commuting with $J$ and $V$ a bounded
$J$-self-adjoint operator anti-commuting with $J$. Since the
involution $J$ is both unitary and self-adjoint, its spectrum
consists of the two points $+1$ and $-1$ and hence
\begin{equation}
\label{JPP} J=\sE_J(\{+1\})-\sE_J(\{-1\}),
\end{equation}
where $\sE_J(\{\pm1\})$ denote the corresponding spectral
projections of $J$. Thus, the involution $J$ induces a natural
decomposition of the Hilbert space $\fH$ into the sum
\begin{equation}
\label{Hsum} \fH=\fH_0\oplus\fH_1
\end{equation}
of the complementary orthogonal subspaces
\begin{equation}
\label{H0H1} \fH_0=\Ran \sE_J(\{+1\}), \quad \fH_1=\Ran
\sE_J(\{-1\}).
\end{equation}
Our assumptions on the operators $A$ and $V$ imply that they are
nothing but the diagonal and off-diagonal parts of $L$ with respect
to the decomposition \eqref{Hsum}:
\begin{alignat}{2}
\label{Adiag} A & = \left(\begin{array}{cc} A_0 & 0\\ 0 &
A_1\end{array}\right),\quad & \dom(A)&=\dom(A_0)\oplus\dom(A_1)
\subset \fH_0\oplus\fH_1, \\
\label{Voff} V & =\left(\begin{array}{cc} 0 & B\\ -B^* &
0\end{array}\right), \quad & \dom(V)&=\fH_0\oplus\fH_1;
\end{alignat}
here the entries $A_0=A\bigr|_{\fH_0}$ and $A_1=A\bigr|_{\fH_1}$ are
self-adjoint operators on $\fH_0$ and $\fH_1$, respectively, and
$B=V\bigl|_{\fH_1}\in\cB(\fH_1,\fH_0)$ is bounded. Thus the operator
$L$ may be viewed as an off-diagonal bounded $J$-self-adjoint
perturbation of the block diagonal self-adjoint operator matrix $A$.

A powerful tool to study operators $L$ admitting a block operator
matrix representation with respect to a self-adjoint involution $J$
is furnished by indefinite inner product spaces. This idea was first
used in \cite{LT2001} to prove a general theorem on
block-diagonalizability for a $J$-accretive operator $A$ and a
self-adjoint perturbation $V$, with application to Dirac operators.
The main ingredient of this approach is to show that the perturbed
reducing subspaces are maximal uniformly positive and negative,
respectively, with respect to the indefinite inner product. As a
consequence, they admit graph representations by angular
operators which measure the deviation between the unperturbed and
the perturbed reducing subspaces.

In the situation considered in the present paper, the self-adjoint
involution $J$ induces an indefinite inner product by means of the
formula
\begin{equation}
\label{IpKs} [x,y]=(Jx,y), \quad x,y\in\fH.
\end{equation}
The Hilbert space $\fH$ equipped with the indefinite inner product
\eqref{IpKs} is a Krein space which we denote by $\fK$, assuming
that $\fK$ stands for the pair $\{\fH,J\}$. Note that every
$J$-self-adjoint operator on $\fH$ is a self-adjoint operator on the
Krein space $\fK$; in particular, the operators $A$, $V$, and
$L=A+V$ are self-adjoint operators on $\fK$. In the Krein space
$\fK$, a (closed) subspace $\fL\subset\fK$ is said to be {uniformly
positive} if there exists a $\gamma>0$ such that
\begin{equation}
\label{mupsIn} [x,x]\geq \gamma\;\|x\|^2 \text{\, \,for every \, }
x\in\fK, \ x\ne 0,
\end{equation}
where $\|\cdot\|$ denotes the norm on $\fH$. The subspace $\fL$ is
called {maximal uniformly positive} if it is not a proper subset of
another uniformly positive subspace of $\fK$. Uniformly negative and
max\-imal uniformly negative subspaces of $\fK$ are defined in a
similar way, replacing the inequality in \eqref{mupsIn} by
$[x,x]\leq -\gamma\;\|x\|^2\!$. Direct sums of subspaces of $\fK$
(or $\fH$) that are $J$-orthogonal (i.e.\ or\-tho\-gonal with
respect to the inner product $[\cdot,\cdot]$) are denoted with
``$[+]$''\!. Further definitions re\-lated to Krein spaces and
linear operators therein may be found, e.g., in \cite{Langer82},
\cite{Bognar}, \cite{IKL}, or~\cite{AI}.

The subspaces $\fH_0$ and $\fH_1$, which simultaneously reduce $A$
and $J$, are maximal uniformly positive and maximal uniformly
negative, respectively, with respect to the inner product
\eqref{IpKs} induced by $J$. Throughout this paper, we assume that
also the perturbed operator $L=A+V$ possesses a maximal uniformly
positive invariant subspace~$\fH_0'$. Then the complementary
$J$-orthogonal subspace $\fH_1'={\fH_0'}^{[\perp]}$ is invariant for $L$
as well; hence both $\fH_0'$ and $\fH_1'$ are automatically reducing
subspaces for $L$ and the spectrum of $L$ is purely real (see, e.g.,
Corollary \ref{LCaux} below).

The main goal of this paper is to establish bounds on the position
of the reducing subspaces $\fH_0'$ or $\fH_1'$ of the perturbed
operator $L=A+V$ relative to the subspaces $\fH_0$ or $\fH_1$. The
bounds are given in terms of the norm of the perturbation $V$ and of
the distances between the unperturbed spectra
\begin{equation}
\label{sigs} \sigma_0=\spec\bigl(A\bigr|_{\fH_0}\bigr)\text{\, and
\,} \sigma_1=\spec\bigl(A\bigr|_{\fH_1}\bigr)
\end{equation}
of $A$ and/or the perturbed spectra
\begin{equation}
\label{sig'} \sigma'_0=\spec\bigl(L\bigr|_{\fH'_0}\bigr)\text{\, and
\,} \sigma'_1=\spec\bigl(L\bigr|_{\fH'_1}\bigr)
\end{equation}
of $L$ in their respective maximal uniformly definite reducing
subspaces.

We describe the mutual geometry of the maximal uniformly definite
reducing subspaces of the unperturbed and perturbed operators $A$
and $L=A+V$ by using the concept of operator angles between two
subspaces of a Hilbert space (for a discussion of this concept and
references see, e.g., \cite{KMM2}). Recall that the operator angle
$\Theta(\fH_i,\fH_i')$ between $\fH_i$ and $\fH_i'$ measured
relative to~$\fH_i$  is given by (see, e.g., \cite{KMM3})
\begin{equation}
\label{ThetaDef}
\Theta_i=\Theta(\fH_i,\fH_i')=\arcsin\sqrt{I_{\fH_i}-P_{\fH_i}
P_{\fH_i'}\bigl|_{\fH_i}}, \quad i=0,1,
\end{equation}
where $I_{\fH_i}$ denotes the identity operator on $\fH_i$, and
$P_{\fH_i}$ and $P_{\fH_i'}$ stand for the orthogonal projections in
$\fH$ onto $\fH_i$ and $\fH_i'$, respectively. By definition, the
operator angle $\Theta(\fH_i,\fH_i')$ is a non-negative operator on
$\fH_i$ and
$$\|\Theta(\fH_i,\fH_i')\|=\max\spec\bigl(\Theta(\fH_i,\fH_i')\bigr)\leq{\pi}/{2}.$$

The main tool we use to estimate the operator angles $\Theta_i=
\Theta(\fH_i,\fH'_i)$ is their relation to solutions of the operator
Riccati equation
\begin{equation}
\label{RicABB} KA_0-A_1K+KBK=-B^*,
\end{equation}
where the coefficients $A_0$, $A_1$, and $B$ are the entries of the
block matrix representations \eqref{Adiag} and \eqref{Voff} of the
operators $A$ and $V$. In fact, given a maximal uniformly positive
reducing subspace $\fH'_0$ of $L=A+V$, there exists a unique
uniformly contractive solution $K$\, ($\|K\|<1$) to the Riccati
equation \eqref{RicABB} such that $\fH'_0$ is the graph of $K$; the
maximal uniformly negative reducing subspace $\fH'_1$ of $L$, which
is $J$-orthogonal to $\fH'_0$, is the graph of the adjoint of $K$.
Since $\|K\|<1$ and $|K|= \tan \Theta(\fH_i,\fH_i')$ (see
Re\-mark~\ref{Rangular} and Lem\-ma~\ref{Lmaxu} below), the operator
angle always satisfies the two-sided inequality
\begin{equation}
\label{Thetapi4} 0\leq\Theta_i<{\pi}/{4},\quad i=0,1.
\end{equation}
By establishing tighter norm bounds on the uniformly contractive
solution $K$ of \eqref{RicABB}, we thus obtain tighter norm bounds
for the operator angles \eqref{ThetaDef}.

Sufficient conditions guaranteeing the existence of maximal
uniformly definite reducing subspaces for the operator $L=A+V$ may
be found, e.g., in \cite{AlMoSh} and \cite{Veselic1,Veselic2}. In
particular, one of the main results of \cite{AlMoSh} is as follows.
Here and in the sequel, by $\conv(\sigma)$ we denote the convex hull
of a Borel set $\sigma\subset\bbR$.

\begin{theorem}[\cite{AlMoSh}, The\-o\-rem~5.8\,(ii)]
\label{IEOT} Assume that the spectral sets $\sigma_0$ and $\sigma_1$
are disjoint, i.e.
\begin{equation}
\label{d} d:=\dist(\sigma_0,\sigma_1)>0,
\end{equation}
that one of these sets lies in a
finite or infinite gap of the other one, i.e.
\[
\conv(\sigma_i)\cap\sigma_{1-i}=\emptyset
\quad \text{or} \quad
\sigma_i\cap\conv(\sigma_{1-i})=\emptyset \quad \text{for some }\ i=0,1,
\]
and that $\,\|V\|<d/2$. Then
$$
\spec(L)=\sigma_0'\,\dot\cup\,\sigma_1',
$$
where the (disjoint) sets $\sigma'_0\subset\bbR$ and
$\sigma'_1\subset\bbR$ lie in the closed
${\|V\|}/{2}$-neighbourhoods of the sets $\sigma_0$ and $\sigma_1$,
respectively. The complementary spectral subspaces $\fH'_0$ and
$\fH'_1$ of $\,L$ associated with the spectral sets $\sigma'_0$ and
$\sigma'_1$ are maximal uniformly positive and maximal uniformly
negative, respectively, and satisfy the sharp norm bound
\begin{equation}
\label{Sin2Tin}
\tan\Theta_i\leq\tanh\left(\frac{1}{2}\arctanh\frac{2\|V\|}{d}\right),
\quad i=0,1,
\end{equation}
or, equivalently,
\begin{equation}
\label{Sin2T1}
\|\sin2\Theta_i\|\leq\frac{2\|V\|}{d}, \quad i=0,1.
\end{equation}
\end{theorem}

The bound \eqref{Sin2Tin} relies on the disjointness of the spectral
sets $\sigma_0$ and $\sigma_1$ of the unperturbed operator $A$ and
involves the distance between $\sigma_0$ and $\sigma_1$. Therefore,
this bound (as well as the other bounds from
\cite[The\-o\-rem~5.8]{AlMoSh}) is an \emph{a priori estimate}. In
the present paper, we establish bounds on the operator angles
$\Theta_i$ that involve at least one of the perturbed spectral sets
$\sigma'_0$ and $\sigma'_1$. In general, for these new bounds to
hold, the disjointness \eqref{d} of the sets $\sigma_0$ and
$\sigma_1$ is not required at all.

Our first main result is a \emph{semi-a posteriori bound} on the
operator angles $\Theta_i$ involving the distances
$\dist(\sigma_0,\sigma'_1)$ and/or $\dist(\sigma_1,\sigma'_0)$
between one unperturbed and one perturbed spectral~set.

\begin{theorem}
\label{ThTheta}
Suppose that $L$ has a maximal uniformly positive reducing subspace
$\fH'_0$ in the Krein space $\fK=\{\fH,J\}$.
\begin{enumerate}
\item[(i)] If for some $i=0,1$ the spectral sets $\sigma_i$
and $\sigma'_{1-i}$ are disjoint, i.e.
\begin{equation}
\label{delIn1} \delta_i:=\dist(\sigma_i,\sigma'_{1-i})>0,
\end{equation}
then the operator angles $\Theta_j$ satisfy the bound
\begin{equation}
\label{TanT0}
\|\tan\Theta_j\|\leq\frac{\pi}{2}\frac{\|V\|}{\delta_i}
\quad\text{for both \,}j=0,1.
\end{equation}
\item[(ii)] If, in addition, one of the sets $\sigma_i$ and
$\sigma'_{1-i}$ satisfying \eqref{delIn1}
lies in a finite or infinite gap of the other one, i.e.
\begin{equation}
\label{disp1} \conv(\sigma_i)\cap\sigma'_{1-i}=\emptyset\text{\, or
\,}\sigma_i\cap\conv(\sigma'_{1-i})=\emptyset,
\end{equation}
then we have the stronger estimate
\begin{equation}
\label{TanT0s} \|\tan\Theta_j\|\leq\frac{\|V\|}{\delta_i} \quad
\text{for both \,} j=0,1.
\end{equation}
\end{enumerate}
\end{theorem}

Our second main result is a completely \emph{a posteriori estimate}
since it only involves the distance between the spectral sets
$\sigma'_0$ and $\sigma'_1$ associated with the perturbed operator
$L$.

\begin{theorem}
\label{ThTgen}
Suppose that $L$ has a maximal uniformly positive reducing subspace
$\fH'_0$ in the Krein space $\fK=\{\fH,J\}$.
\begin{enumerate}
\item[(i)] If the spectral sets $\sigma'_0$ and $\sigma'_1$ are disjoint, i.e.
\begin{equation}
\label{dist'} \widehat{\delta}:=\dist(\sigma'_0,\sigma'_1)>0,
\end{equation}
then the operator angles $\Theta_j$ satisfy the estimate
\begin{equation}
\label{T2TN}
\|\tan\Theta_j\|\leq\frac{\pi}{2}\frac{\|V\|}{\widehat{\delta}}
\quad\text{for both \,}j=0,1.
\end{equation}

\item[(ii)] If,  in addition, for some $i=0,1$ the set $\sigma'_i$
is bounded and lies in a finite or infinite gap of $\sigma'_{1-i}$,
i.e.
\begin{equation}
\label{disp3is} \conv(\sigma'_i)\cap\sigma'_{1-i}=\emptyset,
\end{equation}
then we have the stronger estimate
\begin{equation}
\label{T2TNs}
\|\tan\Theta_j\|\leq\frac{\|V\|}{\sqrt{\widehat{\delta}^2+\|V\|^2}}
\quad \text{ for both \,} j=0,1.
\end{equation}

\item[(iii)] Furthermore, if \,both spectral sets $\sigma'_0$ and
$\sigma'_1$ are bounded and subordinated, i.e.
\begin{equation}
\label{disp2is} \conv(\sigma'_0)\cap\conv(\sigma'_1)=\emptyset,
\end{equation}
then we have the even stronger estimate
\begin{equation}
\label{T2T} \|\tan 2\Theta_j\|\leq\frac{2\|V\|}{\widehat\delta}
\quad \text{ for both \,} j=0,1.
\end{equation}
\end{enumerate}
\end{theorem}

The bounds \eqref{TanT0s} and \eqref{T2T} as well as the bound
\eqref{T2TNs} in the case of a finite gap are optimal (see Remarks
\ref{XXYYY1}--\,\ref{XXYYY3}). Moreover, the sharp a priori
bound \eqref{Sin2Tin} turns out to be a corollary either to Theorem
\ref{ThTheta}\,(ii) or to The\-o\-rem~\ref{ThTgen}\,(iii) (see
The\-o\-rem~\ref{Tbapr-gen} and Re\-mark~\ref{Rlast1}, respectively).

The semi-a posteriori bounds of Theorem~\ref{ThTheta} and the
completely a posteriori ones of Theorem~\ref{ThTgen} complement the
a priori norm bounds on the variation of spectral subspaces for
$J$-self-adjoint operators proved in \cite[Theorem~5.8]{AlMoSh}.
The sharp norm bounds of these theorems represent
analogues of the celebrated trigonometric estimates for self-adjoint
operators known as Davis-Kahan $\sin\Theta$, $\sin2\Theta$,
$\tan\Theta$, and $\tan2\Theta$ theorems (see \cite{DK70} and the
subsequent papers \cite{AlMoSel,KMM3,KMM4,KMM5,MotSel}): the bound
\eqref{Sin2T1} may be called the a priori $\sin2\Theta$ theorem for
$J$-self-adjoint operators; the bounds \eqref{TanT0s} and
\eqref{T2TNs} may be called the semi-a posteriori and completely a
posteriori $\tan\Theta$ theorems, respectively; the bound
\eqref{T2T} may be called the a posteriori $\tan2\Theta$ theorem.

The plan of the paper is as follows. In Sec\-ti\-on~\ref{SecOR} we
give necessary definitions and recall some basic results on the
block diagonalization of $J$-self-adjoint $2\times2$ block operator
matrices. In Sec\-ti\-on~\ref{SecKbounds} we establish several
semi-a posteriori and completely a posteriori norm bounds on
uniformly contractive solutions to operator Riccati equations of the
form \eqref{RicABB}. Using these results, we prove both
The\-o\-rems~\ref{ThTheta} and~\ref{ThTgen} in
Sec\-ti\-on~\ref{SecMain}. Assuming that the spectral sets
\eqref{sigs} do not intersect and
$\|V\|<\frac{1}{2}\dist(\sigma_0,\sigma_1)$, in
Sec\-ti\-on~\ref{SecSpec} we obtain sharp estimates on the position
of the isolated components of the spectrum of $L=A+V$ confined in
the closed $\|V\|$-neighbourhoods of the sets $\sigma_0$ and
$\sigma_1$. In this section, we also establish bounds on the
spectrum for more general $2\times2$ block operator matrices that
need not be $J$-self-adjoint. In Sec\-ti\-on~\ref{SecBapr}, we
combine The\-o\-rems~\ref{ThTheta} and~\ref{ThTgen} with the
spectral estimates of Sec\-ti\-on~\ref{SecSpec} and discuss the
emerging a priori norm bounds on variation of the spectral subspaces
of a self-adjoint operator on a Hilbert space under $J$-self-adjoint
perturbations. Finally, in Sec\-ti\-on~\ref{SecExHO} we apply some
of the bounds obtained to the Schr\"odinger operator describing an
$N$-dimensional isotropic harmonic oscillator under a
$\mathcal{PT}$-symmetric perturbation (see e.g.\ \cite{BenderRep});
here the parity operator $\cP$ plays the role of the self-adjoint
involution $J$ (see \cite{AlbFK,AlMoSh,LT2004-PT}.

The following notations are used thro\-ug\-h\-o\-ut the paper. By a
subspace of a Hilbert space we always mean a closed linear subset.
The identity operator on a subspace (or on the whole Hilbert space)
$\fM$ is denoted by $I_\fM$; if no confusion arises, the index $\fM$
is often omitted. The Banach space of bounded linear operators from
a Hilbert space $\fH$ to a Hilbert space $\fH'$ is denoted by
$\cB(\fH,\fH')$ and by $\cB(\fH)$ if $\fH=\fH'$. The symbol
$\dot\cup$ is used for the union of two disjoint sets. By
$O_r(\Sigma)$, $r\geq 0$, we denote the closed $r$-neigh\-bourhood
of a Borel set $\Sigma$ in the complex plane $\bbC$, i.e.\
$O_r(\Sigma)=\{z\in\bbC\big|\,\dist(z,\Sigma)\leq r\}$. By a finite
gap of a closed Borel set $\sigma\subset\bbR$,
$\sigma\neq\emptyset$, we understand an open interval $(a,b)$,
$-\infty<a<b<\infty$, such that $\sigma\cap(a,b)=\emptyset$ and
$a,b\in\sigma$; an infinite gap of $\sigma$ is a semi-infinite
interval $(a,b)$ such that $\sigma\cap(a,b)=\emptyset$ and either
$a=-\infty$, $|b|<\infty$, and $b\in\sigma$ or $|a|<\infty$,
$a\in\sigma$, and~$b=\infty$.


\section{Preliminaries}
\label{SecOR}

In this section we recall some results on the block diagonalization
of $J$-self-adjoint operator matrices in terms of solutions to the
related operator Riccati equations and on norm bounds for solutions
to operator Sylvester equations. We also recall a couple of
statements on maximal uniformly definite subspaces of a Krein space.
For notational setup we adopt the following

\begin{hypothesis}
\label{Hypo1} Let $J$ be a self-adjoint involution on a Hilbert
space $\fH$, $J\neq I$, and let $\fH_0$ and $\fH_1$ be the spectral
subspaces \eqref{H0H1} of $\,J$. Also assume that $A$ is a (possibly
unbounded) self-adjoint operator on $\fH$ diagonal with respect the
decomposition \eqref{Hsum}, which means that $\fH_0$ and $\fH_1$ are
the reducing subspaces of $A$ and the representation \eqref{Adiag}
holds with $A_0$ and $A_1$ the self-adjoint operators on $\fH_0$ and
$\fH_1$, respectively. Let $V$ be a bounded operator on $\fH$
admitting, relative to \eqref{Hsum}, the representation \eqref{Voff}
where $B\in\cB(\fH_1,\fH_0)$. Finally, let
\begin{equation}
\label{L} L=A+V=\left(\begin{array}{rl} A_0 & B\\ -B^* &
A_1\end{array}\right),\quad\dom(L)=\dom(A).
\end{equation}
\end{hypothesis}

With a block operator matrix $L$ of the form \eqref{L} we associate
the operator Riccati equation \eqref{RicABB} where $K$ is a linear
operator from $\fH_0$ to $\fH_1$. There are different concepts of
solutions to such an equation; here we recall the notion of weak and
strong solutions (see \cite{AMM,AlMoSh}).

\begin{definition}
\label{DefSolRic} Assume that Assumption \ref{Hypo1} is
satisfied. A bounded operator $K\in\cB(\fH_0,\fH_1)$ is said to be
a \emph{weak solution} to the Riccati equation \eqref{RicABB} if
$$
\begin{array}{c}
( KA_0x,y)-( Kx,A_1^*y)+( KBKx,y)=-(B^*x,y) \quad \text{ for all \ }
x\in \dom (A_0), \ y\in \dom(A_1^*).
\end{array}
$$
A bounded operator $K\in\cB(\fH_0,\fH_1)$ is called a \emph{strong
solution} to the Riccati equation \eqref{RicABB} if
\begin{equation}
\label{ranric} \ran\bigl({K}|_{\dom(A_0)}\bigr)\subset\dom(A_1)
\end{equation}
and
\begin{equation}
\label{rics} KA_0x-A_1Kx+KBKx=-B^*x \quad  \text{ for all \ } x\in
\dom(A_0).
\end{equation}
\end{definition}

\begin{remark}
\label{Rweak} Obviously, every strong solution
$K\in\cB(\fH_0,\fH_1)$ to the Riccati equation \eqref{RicABB} is
also a weak solution. In fact, the two notions are equivalent by
\cite[Lemma~5.2]{AM01}: every weak solution $K\in\cB(\fH_0,\fH_1)$
of the operator Riccati equation \eqref{RicABB} is also a strong
solution.
\end{remark}

\begin{remark}
\label{RKKs} With the block operator matrix \eqref{L}, one can also
associate the operator Riccati equation
\begin{equation}
\label{RicABB1} K'A_1-A_0K'-K'B^*K'=B
\end{equation}
where $K'$ is a linear operator from $\fH_1$ to $\fH_0$. {}From
Definition \ref{DefSolRic} it immediately follows that
$K\in\cB(\fH_0,\fH_1)$ is a weak (and hence strong) solution to
\eqref{RicABB} if and only if $K'=K^*$ is a weak (and hence strong)
solution to \eqref{RicABB1}.
\end{remark}

\begin{definition}
\label{DGraph} Let $\fM$ be a subspace of the Hilbert space $\fH$,
$\fM^\perp=\fH\ominus\fM$ its orthogonal complement, and $K$ a
bounded linear operator from $\fM$ to $\fM^\perp$. Denote by $P_\fM$
and $P_{\fM^\perp}$ the orthogonal projections in $\fH$ onto the
subspaces $\fM$ and $\fM^\perp$, respectively. The set
$$
\cG(K)=\{x\in\fH\,|\,\,P_{\fM^\perp}x=KP_\fM x\}
$$
is called the \emph{graph subspace} associated with the operator
$K$.
\end{definition}

\begin{remark}
\label{Rangular} If a subspace $\fG\subset\fH$ is a graph
$\fG=\cG(K)$ of a bounded linear operator
{$K\in\cB(\fM,\fM^\perp)$}, then $K$ is called the
\emph{angular operator} for the (ordered) pair of subspaces $\fM$
and $\fG$; the usage of this term is explained by the equality (see
\cite{KMM2}; cf. \cite{DK70} and \cite{Halmos:69})
\begin{equation}
\label{KThet} |K|=\tan\Theta(\fM,\fG),
\end{equation}
where $|K|$ is the modulus of $K$, $|K|=\sqrt{K^*K}$,  and
$\Theta(\fM,\fG)$ is the operator angle between the subspaces $\fM$
and $\fG$ measured relative to the subspace $\fM$ (see definition
\eqref{ThetaDef}).
\end{remark}

It is well known that strong solutions to the Riccati equations
\eqref{RicABB} and \eqref{RicABB1} determine invariant subspaces for
the operator matrix $L$ by means of their graph subspaces (see,
e.g., \cite{AMM} and \cite{LMMT}). More precisely, the following
correspondences hold (see, e.g., \cite[Lemma~2.4]{AlMoSh}).

\begin{lemma}
\label{Lgraph} Assume that Assumption \ref{Hypo1} holds.
Then the graph $\cG(K)$ of an operator $K\in\cB(\fH_0,\fH_1)$
satisfying \eqref{ranric} is an invariant subspace for the operator
matrix $L$ if and only if $K$ is a strong solution to the operator
Riccati equation \eqref{RicABB}. Similarly, the graph $\cG(K')$ of
an operator 
$K'\in\cB(\fH_1,\fH_0)$ is an invariant
subspace for $L$ if and only if $K'$ is a strong solution to the
Riccati equation \eqref{RicABB1}.
\end{lemma}

The next two statements are well-known facts in the theory of spaces
with indefinite metric (see, e.g., \cite[Section I.8, in particular,
Corollaries I.8.13 and I.8.14]{AI}).

\begin{lemma}
\label{Lmaxu} A subspace $\fL$ is a maximal uniformly positive
subspace of the Krein space $\fK$ if and only if there is a uniform
contraction $K \in \cB(\fH_0,\fH_1)$ (i.e.\ $\|K\|<1$) such that
$\fL$ is the graph $\cG(K)$ of the contraction $K$. Similarly, a
subspace $\fL'$ is a maximal uniformly negative subspace of the
Krein space $\fK$ if and only if  $\fL'$ is the graph $\cG(K')$ of a
uniform contraction $K'\in\cB(\fH_1,\fH_0)$.
\end{lemma}

\begin{lemma}
\label{Lspn} Let $\fL$ be a maximal uniformly positive subspace of
the Krein space $\fK$. Then the orthogonal complement
$\fL^{[\perp]}$ of $\fL$ in $\fK$ is a maximal uniformly negative
subspace. If $\fL$ is a graph subspace, $\fL=\cG(K)$ with $K \in
\cB(\fH_0,\fH_1)$, then $\fL^{[\perp]}$ is the graph of the adjoint
of $K$, i.e.\ $\fL^{[\perp]}=\cG(K^*)$, and
$\fL[+]\fL^{[\perp]}=\fK$.
\end{lemma}

Many more details on Krein spaces and linear operators therein may
be found in \cite{Langer62}, \cite{Langer82}, \cite{Bognar},
\cite{IKL} or \cite{AI}.

The following sufficient condition for a $J$-self-adjoint block
operator matrix of the form \eqref{L} to be similar to a
self-adjoint operator on $\fH$ was proved in \cite{AlMoSh}; for the
particular case where the spectra of the entries $A_0$ and $A_1$ are
subordinated, say $\max\spec(A_0)<\min\spec(A_1)$, closely related
results may be found in \cite[Theorem 4.1]{AdL1995} and
\cite[Theorem 3.2]{MenShk}.

\begin{theorem}[\cite{AlMoSh}, Theorem 5.2]
\label{Lss} Assume that $L=A+V$ satisfies Assumption \ref{Hypo1}.
Suppose that the Riccati equation \eqref{RicABB} has a weak (and
hence strong) solution $K\in \cB(\fH_0,\fH_1)$
such that $\|K\|<1$.
Then:
\begin{enumerate}
\item[(i)]  The operator matrix $L$ has purely real spectrum and it
is similar to a self-adjoint ope\-rator on $\fH$. In particular, the
equality
\begin{equation}
\label{TLam} L=T\Lambda T^{-1}
\end{equation}
holds, where $T$ is a bounded and boundedly invertible operator on
$\fH$ given by
\begin{equation}
\label{Ws} T=\left(\begin{array}{ll}
I & K^* \\
K & I
\end{array}\right)
\left(\begin{array}{cc}
I-K^*K & 0 \\
0 & I-KK^*
\end{array}\right)^{-1/2}
\end{equation}
and $\Lambda$ is a block diagonal self-adjoint operator on $\fH$,
\begin{equation}
\label{Lambda} \Lambda=\diag(\Lambda_0,\Lambda_1), \quad
\dom(\Lambda)=\dom(\Lambda_0)\oplus\dom(\Lambda_1),\,
\end{equation}
whose entries
\begin{equation}
\label{L0p}
\begin{array}{ll}
\Lambda_0=(I-K^*K)^{1/2}(A_0+BK)(I-K^*K)^{-1/2}, \\
\dom(\Lambda_0)=\Ran(I-K^*K)^{1/2}\bigr|_{\dom(A_0)},
\end{array}
\end{equation}
and
\begin{equation}
\label{L01}
\begin{array}{ll}
\Lambda_1=(I-KK^*)^{1/2}(A_1-B^*K^*)(I-KK^*)^{-1/2}, \\
\dom(\Lambda_1)=\Ran(I-KK^*)^{1/2}\bigr|_{\dom(A_1)},
\end{array}
\end{equation}
are self-adjoint operators on the corresponding  Hilbert space
components $\fH_0$ and $\fH_1$, respectively. \smallskip

\item[(ii)] The graph subspaces $\fH'_0=\cG(K)$ and $\fH'_1=\cG(K^*)$
are invariant under $L$, mutually orthogonal with respect to the
indefinite inner product \eqref{IpKs}, and
$$
\fK=\fH'_0[+]\fH'_1.
$$
The subspace $\fH'_0$ is  maximal uniformly positive, while $\fH'_1$
is maximal uniformly negative. The restrictions of $L$ onto $\fH'_0$
and $\fH'_1$ are $\fK$-unitary equivalent to the self-adjoint
operators $\Lambda_0$ and $\Lambda_1$, respectively.
\end{enumerate}
\end{theorem}

\begin{remark}
\label{RKcontr} The requirement $\|K\|<1$ is sharp in the sense that
if there is no uniformly contractive solution to the Riccati
equation \eqref{RicABB}, then the operator matrix $L$ need not be
similar to a self-adjoint operator at all; this can be seen, e.g.,
from \cite[Example 5.5]{AlMoSh}.
\end{remark}

An elementary consequence of Theorem \ref{Lss} is the following
property of maximal uniformly definite subspaces of $J$-self-adjoint
operators $L=A+V$ with self-adjoint $A$ and bounded $V$.

\begin{corollary}
\label{LCaux}
Assume that $L=A+V$ satisfies Assumption \ref{Hypo1}. Suppose that
$L$ has a maximal uniformly positive $($resp.\ negative$)$ invariant
subspace $\fK_0$ of $\fK=\{\fH,J\}$. Then $\fK_1=\fK_0^{[\perp]}$ is
also an invariant subspace of $\,L$, which is maximal uniformly
negative $($resp.\ positive$)$;
the restrictions of $\,L$ to $\fK_0$ and $\fK_1$ are $\fK$-unitary
equivalent to self-adjoint operators on the Hilbert spaces $\fH_0$
and $\fH_1$, respectively.
\end{corollary}

\begin{proof}
We give the proof for the case where $\fK_0$ is a maximal uniformly
positive subspace; the proof for maximal uniformly negative $\fK_0$
is analogous.

By Lemma \ref{Lmaxu}, $\fK_0$ is the graph of a uniform contraction
$K:\,\fH_0\to\fH_1$. Since $\fK_0$ is invariant under $L$, Lemma
\ref{Lgraph} shows that $K$ is a uniformly contractive strong
solution to the Riccati equation \eqref{RicABB}.
Now  all claims follow immediately from Theorem \ref{Lss}.
\end{proof}

Riccati equations are closely related to operator Sylvester
equations (also called Kato-Rosen\-blum equations). In this
paper we use the following well known result on sharp norm bounds
for strong solutions to operator Sylvester equations (cf.\
\cite[The\-o\-rem~4.9]{AlMoSh}).

\begin{theorem}
\label{TSylB} Let $A_0$ and $A_1$ be (possibly unbounded)
self-adjoint operators on the Hilbert spaces $\fH_0$ and $\fH_1$,
respectively, and $Y \in \cB(\fH_0, \fH_1)$. If the spectra
$\spec(A_0)$ and $\spec(A_1)$ are disjoint, i.e.\
$$d:=\dist\bigl(\spec(A_0),\spec(A_1)\bigr) > 0,$$
then the operator Sylvester equation
$$
XA_0-A_1X=Y
$$
has a unique strong solution $X\in\cB(\fH_0,\fH_1)$; the solution
$X$ satisfies the norm bound
\begin{equation}
\label{XBg} \|X\|\leq\frac{\pi}{2}\frac{\|Y\|}{d};
\end{equation}
if, in addition, one of the sets $\spec(A_0)$ and $\spec(A_1)$ lies
in a finite or infinite gap of the other one, then $X$ satisfies the
stronger norm bound
\begin{equation}
\label{XBs} \|X\|\leq\frac{\|Y\|}{d}.
\end{equation}
\end{theorem}

\begin{remark}
The fact that the constant $\pi/2$ in the estimate \eqref{XBg} for
the generic disposition of the sets $\spec(A_0)$ and $\spec(A_1)$ is
best possible is due to R. McEachin \cite{McE93}. The existence of
the bound \eqref{XBs} for the particular case where one of the sets
$\spec(A_0)$ and $\spec(A_1)$ lies in a finite or infinite gap of
the other one may be traced back to E.\,Heinz \cite{H51} (also see
\cite[Theorem 3.2]{BDM1983} and \cite[Theorem 3.4]{AlMoSh}). For
more details and references we refer the reader to
\cite[Re\-mark~2.8]{AMM} and \cite[Re\-mark~4.10]{AlMoSh}.
\end{remark}

In the proofs of several statements below we will use the following
elementary result, the proof of which is left to the reader.

\begin{lemma}
\label{phiT} Let $\varphi$ be a scalar analytic function of a
complex variable $z$ whose Taylor~series
\begin{equation}
\label{phiTay} \varphi(z)=\sum\limits_{k=0}^\infty a_k z^k, \quad
a_k=\frac{1}{k!}\dfrac{d^k\varphi(0)}{dz^k}, \quad k=1,2,\ldots,
\end{equation}
is absolutely convergent on the open disc $\{z\in\bbC:|z|<r\}$ for
some $r>0$. Let  $M\in\cB(\fH_1,\fH_0)$ and $N\in\cB(\fH_0,\fH_1)$
be bounded operators with $\|MN\|<r$ and $\|NM\|<r$. Then
\begin{equation}
\label{MphiN} M\varphi(NM)=\varphi(MN)M,
\end{equation}
where for a bounded linear operator $T$ on a Hilbert space $\fT$
with $\|T\|<r$ the value of $\varphi(T)$ is defined by the series
$$
\varphi(T)=\sum\limits_{k=0}^\infty a_k T^k.
$$
\end{lemma}

We also need the following auxiliary statement.

\begin{lemma}
\label{LRKK} Assume that Assumption \ref{Hypo1} holds and
suppose that the Riccati equation \eqref{RicABB} has a weak (and
hence strong) solution $K\in\cB(\fH_1,\fH_0)$ such that $\|K\|<1$.
Then
\begin{equation}
\label{DomKL}
\Ran\bigl(K\bigl|_{\Dom(\Lambda_0)}\bigr)\subset\Dom(\Lambda_1)
\end{equation}
and
\begin{equation}
\label{Ric-t2t}
K\Lambda_0y-\Lambda_1Ky=-(I-KK^*)^{1/2}B^*(I-K^*K)^{1/2}y\quad
\text{for all }y\in\Dom(\Lambda_0),
\end{equation}
where $\Lambda_0$ and $\Lambda_1$ are the self-adjoint operators
given by \eqref{L0p} and \eqref{L01}, respectively.
\end{lemma}

\begin{proof}
It is straightforward to verify that if $K$ is a strong solution to
the Riccati equation \eqref{RicABB}, then
\begin{equation}
\label{RKs} KZ_0x-Z_1Kx=-B^*(I-K^*K)x\quad \text{for all
}x\in\Dom(A_0).
\end{equation}
Since $K$ is assumed to be a uniform contraction, Theorem
\ref{Lss}\,(i) applies and yields
\begin{align}
\nonumber
&K(I-K^*K)^{-1/2}\Lambda_0(I-K^*K)^{1/2}x-(I-KK^*)^{-1/2}\Lambda_1
(I-KK^*)^{1/2}Kx\qquad\\
\label{ZLA}
&\qquad\qquad\qquad =-B^*(I-K^*K)x \qquad
\text{for all \ } x\in\dom(A_0).
\end{align}
Applying $(I-KK^*)^{1/2}$ from the left to both sides of \eqref{ZLA}
and choosing $x=(I-K^*K)^{-1/2}y$ with $y\in\Dom(\Lambda_0)$, we
arrive at the Sylvester equation
\begin{equation}
\label{XKSyl} X\Lambda_0y-\Lambda_1 Xy=Yy \quad \text{for all }
y\in\dom(\Lambda_0),
\end{equation}
where
\begin{align}
\label{XX}
X=&\ \, (I-KK^*)^{1/2}K(I-K^*K)^{-1/2},\qquad\\
\label{XY} Y=&-(I-KK^*)^{1/2}B^*(I-K^*K)^{1/2}.
\end{align}
Note that we have
$\Ran(I-K^*K)^{-1/2}\bigl|_{\Dom(\Lambda_0)}=\Dom(A_0)$ by
\eqref{L0p}, $\ran\bigl({K}|_{\dom(A_0)}\bigr)\subset\dom(A_1)$ by
\eqref{ranric}, and thus, by \eqref{L01},
\begin{equation}
\label{DomX}
\Ran\bigl(X\bigr|_{\dom(\Lambda_0)}\bigr)\subset\dom(\Lambda_1).
\end{equation}
Hence $X$ is a strong solution to the Sylvester equation
\eqref{XKSyl}.

Furthermore, the Taylor series \eqref{phiTay} of the function
$\varphi(z)=(1-z)^{1/2}$ is absolutely convergent on the disc
$\{z\in\bbC:|z|<1\}$. Since $\|K\|<1$, Lemma \ref{phiT} applies and
yields that $(I-KK^*)^{1/2}K=K(I-K^*K)^{1/2}$. Therefore, \eqref{XX}
simplifies to nothing but the identity
$X=K$.
Now the claims follow from the inclusion \eqref{DomX} and the
identities \eqref{XKSyl}, \eqref{XY}.
\end{proof}


\section{Bounds on uniformly contractive solutions to the Riccati equations}
\label{SecKbounds}

Assuming Assumption \ref{Hypo1}, in this section we prove several
norm bounds on uniformly contractive solutions $K$ to the Riccati
equation \eqref{RicABB} (provided such solutions exist). These
bounds are obtained under the hypothesis that either the
spectra of the operators $Z_0=A_0+BK$ and $A_1$ or the spectra of
$Z_0$ and $Z_1=A_1-B^*K^*$ are disjoint. Note that, by Theorem
\ref{Lss}\,(i), the assumption $\|K\|<1$ implies that the spectra of
$Z_0$ and $Z_1$ are both real, that is, $\spec(Z_0)\subset\bbR$ and
$\spec(Z_1)\subset\bbR$.

Throughout this section we use the following notations:
\begin{align}
\label{delZA}
\delta_{Z_0,A_1}:=&\dist\bigl(\spec(Z_0),\spec(A_1)\bigr), \\
\label{delZZ}
\delta_{Z_0,Z_1}:=&\dist\bigl(\spec(Z_0),\spec(Z_1)\bigr).
\end{align}

\subsection{Semi-a posteriori bounds}
\label{SecSAPB} First, we establish norm bounds on $K$ that only
contain the norm of $B$ and the distance $\delta_{Z_0,A_1}$.
Therefore, these bounds may be viewed as semi-a posteriori estimates
on $K$ since the set $\spec(Z_0)=\spec(\Lambda_0)$ corresponds to
the perturbed operator $L=A+V$ (see Theorem \ref{Lss}), while the
other set, $\spec(A_1)$, is part of the spectrum of the
unperturbed ope\-ra\-tor $A$.

\begin{theorem}
\label{Tmain} Assume that Assumption \ref{Hypo1} holds
and suppose that the Riccati equation \eqref{RicABB} has a weak (and
hence strong) solution $K\in\cB(\fH_1,\fH_0)$ such that $\|K\|<1$.
Then:
\begin{enumerate}
\item[(i)] If the spectra of the operators $Z_0=A_0+BK$, $\dom(Z_0)=\dom(A_0)$,
and of $A_1$ are disjoint, i.e.
\begin{equation}
\label{ddelta} \delta_{Z_0,A_1}>0,
\end{equation}
then the solution $K$ satisfies the inequality
\begin{equation}
\label{Kpid}
\|K\|\leq\frac{\pi}{2}\,\frac{\|B\|}{\quad\delta_{Z_0,A_1}}.
\end{equation}
\item[(ii)] If, in addition, one of the sets $\spec(Z_0)$ or $\spec(A_1)$
lies in a finite or infinite gap of the other one, i.e.
\begin{equation}
\label{Z0A1} \conv\bigl(\spec(Z_0)\bigr)\cap\spec(A_1)=\emptyset
\end{equation}
or
\begin{equation}
\label{A1Z0} \spec(Z_0)\cap\conv\bigl(\spec(A_1)\bigr)=\emptyset,
\end{equation}
then the solution $K$ satisfies the stronger inequality
\begin{equation}
\label{K1d} \|K\|\leq\frac{\|B\|}{\quad\delta_{Z_0,A_1}}.
\end{equation}
\end{enumerate}
\end{theorem}

\begin{proof}
The assumption that $K$ is a strong solution to the Riccati equation
\eqref{RicABB} is equivalent to $\ran\bigl({K}|_{\dom(Z_0)}\bigr)=
\ran\bigl({K}|_{\dom(A_0)}\bigr)\subset\dom(A_1)$ and
\begin{equation}
\label{KZ} KZ_0x-A_1Kx=-B^*x\quad \text{for all \
}x\in\dom(A_0)=\dom(Z_0).
\end{equation}
Since $K$ is a uniform contraction, $\|K\|<1$, we can use Theorem
\ref{Lss}\,(i) to rewrite \eqref{KZ} as
\begin{equation}
\label{ZL} K(I-K^*K)^{-1/2}\Lambda_0(I-K^*K)^{1/2}x-A_1Kx=-B^*x\quad
\text{for all }x\in\dom(A_0),
\end{equation}
where $\Lambda_0$ is the self-adjoint operator defined by
\eqref{L0p}. If we choose $x=(I-K^*K)^{-1/2}y$ with
$y\in\dom(\Lambda_0) $, we can write \eqref{ZL} as
\begin{align}
\label{ZL1}
K(I-K^*K)^{-1/2}\Lambda_0y-A_1K(I-K^*K)^{-1/2}y=-B^*(I-K^*K)^{-1/2}y\\
\nonumber \qquad\text{for all \ }y\in\dom(\Lambda_0);
\end{align}
note that $\ran
K(I-K^*K)^{-1/2}\bigl|_{\dom(\Lambda_0)}\subset\dom(A_1)$ since
$\Ran(I-K^*K)^{1/2}\bigr|_{\dom(A_0)}=\dom(\Lambda_0)$
(see\eqref{L0p}) and $\ran\bigl(
K\bigl|_{\dom(A_0)}\bigr)\subset\dom(A_1)$. Equality \eqref{ZL1}
means that the operator
\begin{equation}
\label{XK} X=K(I-K^*K)^{-1/2}
\end{equation}
is a strong solution to the operator Sylvester equation
\begin{equation}
\label{XSyl} X\Lambda_0-A_1X=Y
\end{equation}
with $Y=-B^*(I-K^*K)^{-1/2}$. Obviously, for the norm of $Y$ we have
\begin{equation}
\label{YB} \|Y\|\leq\frac{\|B\|}{\sqrt{1-\|K\|^2}}.
\end{equation}
If $|K|=\sqrt{K^*K}$ denotes the modulus of $K$, then
the modulus $|X|=\sqrt{X^*X}$ of the operator $X$ defined in
\eqref{XK} is given by
$$
|X|=|K|(I-|K|^2)^{1/2}.
$$
Taking into account that $\bigl\||X|\bigr\|=\|X\|$ and
$\bigl\||K|\bigr\|=\|K\|$, the spectral theorem implies that
\begin{equation}
\label{Xnorm} \|X\|=\frac{\|K\|}{\sqrt{1-\|K\|^2}}.
\end{equation}
Due to the similarity \eqref{L0p} of the operators $\Lambda_0$ and
$Z_0$, we have $\spec(\Lambda_0)=\spec(Z_0)$ and hence,
by~\eqref{ddelta},
\begin{equation}
\label{delL}
\dist\bigl(\spec(\Lambda_0),\spec(A_1)\bigr)=\delta_{Z_0,A_1}.
\end{equation}
Applying Theorem \ref{TSylB} and using \eqref{YB} as well as
\eqref{Xnorm}, we readily arrive at
$$
\frac{\|K\|}{\sqrt{1-\|K\|^2}}=\|X\|\leq c
\,\frac{\|Y\|}{\delta_{Z_0,A_1}}\leq
c\,\frac{\|B\|}{\delta_{Z_0,A_1}\sqrt{1-\|K\|^2}}
$$
where $c=\pi/2$ in case (i) and $c=1$ in case (ii) so that, in both
cases,
$$
\|K\|\leq c \, \frac{\|B\|}{\delta_{Z_0,A_1}}. \vspace{-8.5mm}
$$
\end{proof}

\vspace{2mm}

\begin{remark}
In order to compete with the hypothesis $\|K\|<1$, the bounds
\eqref{Kpid} and \eqref{K1d} are of interest only if $\|B\|< 2
\,\delta_{Z_0,A_1} / \pi$ in case (i) and $\|B\|<\delta_{Z_0,A_1}$
in case (ii).
\end{remark}

\begin{remark}
\label{REx1} For all spectral dispositions such that \eqref{Z0A1} or
\eqref{A1Z0} holds and $\|B\|<\delta_{Z_0,A_1}$, the bound \eqref{K1d} is
sharp in the sense that given an arbitrary $\beta>0$ and arbitrary
$\delta>\beta$ one can always find $A$ and $V$ such that
$\|V\|=\beta$, $\delta_{Z_0,A_1}=\delta$, and $\|K\|=\beta/\delta$.
\end{remark}

The following examples illustrate the sharpness of \eqref{K1d} and
Remark \ref{REx1}.

\begin{example}
\label{Ex3} Let $\fH_0=\bbC^2$ and $\fH_1=\bbC$. Assume that
$b,d\in\bbR$ are such that $0\leq b<d/2$ and~let
$$
A_0=\left(\begin{array}{rr} -d & 0\\0 &d
\end{array}\right), \quad A_1=0, \quad B=\left(\begin{array}{c}
0 \\ b\end{array}\right).
$$
For this choice of $A_0$, $A_1$, and $B$, the Riccati equation
\eqref{RicABB} has a unique uniformly contractive solution of the
form $K=\bigl( 0 \ -\kappa \bigr)$ where $\kappa$ is given by
\begin{equation}
\label{K12} \kappa=\dfrac{b}{\frac{d}{2}+\sqrt{\frac{d^2}{4}-b^2}}.
\end{equation}
Hence
$$
Z_0=A_0+BK=\left(\begin{array}{cc} -d & 0\\0 &
\frac{d}{2}+\sqrt{\frac{d^2}{4}-b^2}
\end{array}\right)
$$
and the set $\spec(A_1)=\{0\}$ lies in the gap $(-d,
d/2+\sqrt{d^2/4-b^2})$ of $\spec(Z_0)$, so that \eqref{A1Z0} holds.
Altogether we have
\begin{equation}
\label{KsharpZ0A1-3isl} \|K\|=\frac{\|B\|}{\delta_{Z_0,A_1}},
\end{equation}
i.e.\ equality in \eqref{K1d}.
\end{example}

\begin{example}
\label{Ex1} Let $\fH_0=\bbC$ and $\fH_1=\bbC^2$. Assume that
$b,d\in\bbR$ are such that $0\leq b<d$ and set
$$
A_0=0,\quad A_1=\left(\begin{array}{rr} -d & 0\\0 &d
\end{array}\right), \quad B=\left(\begin{array}{lr}
\frac{b}{\sqrt{2}} & \frac{b}{\sqrt{2}}\end{array}\right).
$$
By inspection, one can verify that the $2\times 1$ matrix
$$
K=\left(\begin{array}{r}-\frac{b}{\sqrt{2}d}\\
\frac{b}{\sqrt{2}d}\end{array}\right)
$$
is a solution to the operator Riccati equation \eqref{RicABB}.
Clearly,
\begin{equation}
\label{BbK} \|B\|=b, \quad  \|K\|=\dfrac{b}{d}
\end{equation}
and
$$
Z_0=A_0+BK=0,\quad Z_1=A_1-B^*K^*=\left(\begin{array}{cc}
-d+\frac{b^2}{2d} & -\frac{b^2}{2d}\\
\frac{b^2}{2d}   &  \quad d-\frac{b^2}{2d}
\end{array}\right).
$$
Obviously, the set $\spec(Z_0)=\{0\}$ lies within the gap $(-d,d)$
of the set $\spec(A_1)=\{-d,d\}$. Furthermore, $\delta_{Z_0,A_1}=d$
and hence, by \eqref{BbK},
\begin{equation}
\label{KBdel} \|K\|=\dfrac{\|B\|}{\delta_{Z_0,A_1}}.
\end{equation}
For later reference, we note that
$\spec(Z_1)=\big\{-\sqrt{d^2-b^2},\sqrt{d^2-b^2}\big\}$ so that
$\delta_{Z_0,Z_1}=\sqrt{d^2-b^2}$ and thus
\begin{equation}
\label{sharp3} \|K\|=\frac{\|B\|}{\sqrt{\delta_{Z_0,Z_1}^2+b^2}}=
\frac{\|B\|}{\sqrt{\delta_{Z_0,Z_1}^2+\|B\|^2}}.
\end{equation}
\end{example}

\begin{example}
\label{Ex2} Let $\fH_0=\fH_1=\bbC$.
Assume that $b,d\in\bbR$ are such that $0< b<d/2$ and set
\[
  A_0=-d/2, \quad A_1=d/2, \quad B=b.
\]
Then the Riccati equation \eqref{RicABB} appears to be the numeric
quadratic equation $b K^2 + K d = -b$. The only solution $K=\kappa
\in \bbR$ with norm $\|K\|=|\kappa|<1$ where $\kappa$ is again given
by \eqref{K12}. One immediately verifies that
\begin{equation}
\label{Ex2Z01} Z_0=A_0+BK=-\sqrt{\frac{d^2}{4}-b^2},\quad
Z_1=A_1-B^*K^*=\sqrt{\frac{d^2}{4}-b^2}.
\end{equation}
Here the sets $\spec(Z_0)$ and $\spec(A_1)$ are even subordinated to
each other, so that both \eqref{Z0A1} and \eqref{A1Z0} hold.
Together with \eqref{K12} and \eqref{Ex2Z01}, we again obtain the
equality
\begin{equation}
\label{KsharpZ0A1} \|K\|=\frac{\|B\|}{\delta_{Z_0,A_1}}.
\end{equation}
For later reference, we also observe that
\begin{equation}
\label{Ksharp-t2t}
\frac{\|K\|}{1-\|K\|^2}=\frac{\|B\|}{\delta_{Z_0,Z_1}}.
\end{equation}
\end{example}

\subsection{Completely a posteriori bounds}
\label{SecCAPB} In this subsection we consider the case where the
spectra of the operators $Z_0=A_0+BK$ and $Z_1=A_1-B^*K^*$ are
disjoint. The bounds on $K$ obtained here depend only on
$\|B\|$ and on the distance $\delta_{Z_0,Z_1}$ between the
subsets $\spec(Z_0)=\spec(\Lambda_0)$ and
$\spec(Z_1)=\spec(\Lambda_1)$ of the spectrum of the perturbed
operator $L=A+V$ (see The\-o\-rem~\ref{Lss}). Therefore, they may be
viewed as a posteriori bounds on $K$.

We begin with the most general result where nothing is known on the
mutual position of $\spec(Z_0)$ and $\spec(Z_1)$ except that they do not intersect.

\begin{theorem}
\label{Ttanp} Assume that $L=A+V$ satisfies Assumption \ref{Hypo1}
and suppose that the Riccati equation \eqref{RicABB} has a weak (and
hence strong) solution $K\in\cB(\fH_1,\fH_0)$ such that $\|K\|<1$.
If the~spectra of the operators $Z_0=A_0+BK$, $\dom(Z_0)=\dom(A_1)$,
and $Z_1=A_1-B^*K^*$, $\dom(Z_1)=\dom(A_1),$ do not intersect, that
is,
$$\delta_{Z_0,Z_1}>0,$$
then
\begin{equation}
\label{AKpid} \|K\|\leq\frac{\pi}{2}\frac{\|B\|}{\delta_{Z_0,Z_1}}.
\end{equation}
\end{theorem}

\begin{proof}
By Lemma \ref{LRKK}, the Riccati equation \eqref{RicABB} can be
written in the form \eqref{Ric-t2t}. For the term
$Y=-(I-KK^*)^{1/2}B^*(I-K^*K)^{1/2}$ on the right-hand side of
\eqref{Ric-t2t}, we have $\|Y\|\leq \|B\|$ since both $K^*K$ and
$KK^*$ are non-negative and, in addition, $\|K\|\|K^*\|<1$. Since
$\spec(\Lambda_0)=\spec(Z_0)$ and $\spec(\Lambda_1)=\spec(Z_1)$, we
have
$\dist\bigl(\spec(\Lambda_0),\spec(\Lambda_1)\bigr)=\delta_{Z_0,Z_1}$.
To complete the proof, it remains to apply the bound \eqref{XBg}
from Theorem \ref{TSylB} to \eqref{Ric-t2t}.
\end{proof}

\begin{remark}
Under the stronger assumption that one of the spectral sets
$\spec(Z_0)$ and $\spec(Z_1)$ lies in a finite or infinite gap of
the other one, i.e.\ if
\begin{equation*}
\conv \bigl(\spec(Z_i)\bigr) \cap \spec(Z_{1-i})=\emptyset \text{\,
for some \,}i=0,1,
\end{equation*}
Theorem \ref{TSylB} also yields the estimate
\begin{equation}
\label{AK1del} \|K\|\leq\frac{\|B\|}{\delta_{Z_0,Z_1}}.
\end{equation}
This estimate, however, is of no interest in the case where the
corresponding operator $Z_i$ is bounded: the bound
\eqref{KZconv} in the following theorem is stronger than \eqref{AK1del}.
\end{remark}

\begin{theorem}
\label{Ttans} Assume that $L=A+V$ satisfies Assumption \ref{Hypo1}
and suppose, in addition, that $A_0$ is bounded. Let the Riccati
equation \eqref{RicABB} have a weak (and hence strong) solution
$K\in\cB(\fH_1,\fH_0)$ such that $\|K\|<1$. If the spectrum of the
(bounded) operator $Z_0=A_0+BK$ lies in a finite or infinite gap of
the spectrum of the operator $Z_1=A_1-B^*K^*$,
$\dom(Z_1)=\Dom(A_1)$, that is,
\begin{equation}
\label{Z01conv} \conv\bigl(\spec(Z_0)\bigr)\cap\spec(Z_1)=\emptyset,
\end{equation}
then
\begin{equation}
\label{KZconv}
\|K\|\leq\frac{\|B\|}{\sqrt{\delta_{Z_0,Z_1}^2+\|B\|^2}}.
\end{equation}
\end{theorem}

\begin{proof}
Throughout the proof we assume that $B\neq0$ and, thus, necessarily
\begin{equation}
\label{Kne0}
K\neq 0.
\end{equation}
Let $U$ be the partial isometry in the polar
decomposition $K=U|K|$ of $K$. If we adopt the convention that $U$
is extended to $\Ker(K)=\Ker(|K|)$ by
\begin{equation}
\label{conviso} U|_{\Ker(K)}=0,
\end{equation}
then $U$ is uniquely defined on the whole space $\fH_0$
(see \cite[Theorem 8.1.2]{BirSol} or \cite[\S VI.7.2]{K})~and
\begin{equation}
\label{Kisom} U \text{ \ is an isometry on \ } \Ran(|K|)=\Ran(K^*).
\end{equation}

First we apply Lemma \ref{LRKK} and transform the Riccati equation
\eqref{RicABB} to the form \eqref{Ric-t2t}. Since the operator
$\Lambda_0$ is bounded, $\Lambda_0\in\cB(\fH_0)$, we may then
rewrite \eqref{Ric-t2t} as
\begin{equation}
\label{Ric-3i}
K\Lambda_0-\Lambda_1U|K|=-(I-KK^*)^{1/2}B^*(I-|K|^2)^{1/2} = -
{\widetilde{B}}^*(I-|K|^2)^{1/2}
\end{equation}
where we have set
\begin{equation}
\label{Bt} \widetilde{B}=B(I-KK^*)^{1/2}.
\end{equation}

We tackle the cases where the gap of $\spec(\Lambda_1)$ containing
the set $\spec(\Lambda_0)$ is finite or infinite in a slightly
different manner. If this gap is finite, we may assume without loss
of generality that it is centered at zero, i.e.\ it is of the form
$(-a,a)$ with $a>0$; otherwise, we simply replace $\Lambda_0$ and
$\Lambda_1$ in \eqref{Ric-3i} by $\Lambda'_0=\Lambda_0-\lambda_1 I$
and $\Lambda'_1=\Lambda_1-\lambda_1 I$, respectively, where
$\lambda_1$ is the center of the gap. Then
\begin{equation}
\label{LaLa} 0\in \rho(\Lambda_1), \quad
\|\Lambda_1^{-1}\|<\frac{1}{a},\text{\, and \,}\|\Lambda_0\|\leq
a-\delta_{Z_0,Z_1}.
\end{equation}
If the gap of $\spec(\Lambda_1)$ containing the $\spec(\Lambda_0)$
is infinite, we may assume without loss of generality that the
interval $[\min\spec(\Lambda_0),\max\spec(\Lambda_0)]$ is centered
at zero and that the spectrum of $\Lambda_1$ lies either in the
interval $(-\infty,-a]$ where $a=-\max\spec(\Lambda_1)$ or in the
interval $[a,\infty)$ where $a=\min\spec(\Lambda_1)$. Then, again,
all three statements of \eqref{LaLa} hold.

In the following, we may thus treat the two above cases together.
Since $0\not\in\spec(\Lambda_1)$, we further rewrite \eqref{Ric-3i}
in the form
\begin{equation}
\label{Ric-Bt}
U|K|=\Lambda_1^{-1}\left(K\Lambda_0+{\widetilde{B}}^*(I-|K|^2)^{1/2}\right)
\end{equation}
and set $\kappa=\max\,\spec(|K|)=\|K\|$. By assumption
\eqref{Kne0}, we have $\kappa>0$.

If $\kappa$ is an eigenvalue of $|K|$ and $x$ a corresponding
eigenvector with $\|x\|=1$, then, by applying both sides of
\eqref{Ric-Bt} to $x$, we immediately arrive at
\begin{equation}
\label{x3is} \kappa\,Ux=\Lambda_1^{-1} \bigl( K\Lambda_0
x+\sqrt{I-\kappa^2}\,\widetilde{B}^*x \bigr).
\end{equation}
By \eqref{Bt}, we have $\|\widetilde{B}^*x\|\le \|\widetilde{B}^*\|
\le \|B\| \|(I-KK^*)^{1/2}\| \leq\|B\|$. Then, by \eqref{LaLa} and
\eqref{x3is}, we obtain
\begin{equation}
\label{kb3s} \kappa\leq\frac{1}{a}\bigl(\kappa(a-\delta_{Z_0,Z_1})+
\sqrt{1-\kappa^2}\,\|B\|\bigr),
\end{equation}
taking into account that $x\in\Ran(|K|)$ and thus $\|Ux\|=\|x\|=1$
by \eqref{Kisom}.

If $\kappa$ is not an eigenvalue of $|K|$, then it belongs to the
essential spectrum of $|K|$.
Hence we obtain a singular sequence
$\{x_n\}_{n=1}^\infty$ of $|K|$ at $\kappa$ by choosing arbitrary
\begin{equation}
\label{xn} x_n\in
\Ran\sE_{|K|}\bigl((\kappa(1-1/n),\kappa]\bigr),\quad \|x_n\|=1,
\quad n=1,2,\ldots;
\end{equation}
here $\sE_{|K|}$ denotes the spectral measure of $|K|$ and, at the
same time, the (right-continuous) spectral function of $|K|$, that
is, $\sE_{|K|}(\mu)=\sE_{|K|}\bigl((-\infty,\mu]\bigr)$. Obviously,
$|K|x_n=\kappa x_n +\varepsilon_n$ with
\begin{equation}
\label{epsn}
\|\varepsilon_n\|=\biggl\|\int_{\kappa\left(1-\frac{1}{n}
\right)}^\kappa
d\sE_{|K|}(\mu)(\mu-\kappa)x_n\biggr\|\leq\frac{1}{n}.
\end{equation}
Similarly, $(I-|K|^2)^{1/2}x_n=\sqrt{1-\kappa^2}\,x_n+\zeta_n$ with
\begin{align}
\nonumber \|\zeta_n\|&=\biggl\|\int_{\kappa\left(1-\frac{1}{n}
\right)}^\kappa
d\sE_{|K|}(\mu)\bigl(\sqrt{1-\mu^2}-\sqrt{1-\kappa^2}\bigr)x_n\biggr\|\\
\nonumber &\leq\sup\limits_{\mu\in\left(\kappa\left(1-{1}/{n}
\right),\kappa\right]}\bigl|\sqrt{1-\mu^2}-\sqrt{1-\kappa^2}\bigr|\\
\label{zetn}
&=\left(1-\kappa^2\left(1-\frac{1}{n}\right)^2\right)^{1/2}
-(1-\kappa^2)^{1/2} \ <\
\frac{1}{n}\,\frac{\kappa^2}{\sqrt{1-\kappa^2}}.
\end{align}

Applying both sides of equality \eqref{Ric-Bt} to $x_n$, we arrive
at
\begin{equation}
\label{kb3sa} \kappa \,Ux_n=\Lambda_1^{-1} \bigl( K\Lambda_0 x_n+
\sqrt{I-\kappa^2}\,\widetilde{B}^*x_n \bigr) +\beta_n
\end{equation}
with
\begin{equation}
\label{betan}
\beta_n=\Lambda_1^{-1}\widetilde{B}^*\zeta_n-U\epsilon_n \to 0
\text{\, as \,} n\to\infty;
\end{equation}
here we have used that $\epsilon_n=|K|x_n-\kappa x_n \to 0$ and
$\zeta_n=(I-|K|^2)^{1/2}x_n-\sqrt{1-\kappa^2}x_n \to 0$ as
$n\to\infty$ by \eqref{epsn} and \eqref{zetn}, respectively. Since
$x_n \in \Ran(|K|)$ and thus $\|Ux_n\|=\|x_n\|=1$ by \eqref{Kisom}, the
relation \eqref{kb3sa} implies that
\begin{equation}
\label{kb3sb} \kappa\leq\frac{1}{a}\bigl(\kappa(a-\delta_{Z_0,Z_1})+
\sqrt{1-\kappa^2}\|B\|\bigr)+\|\beta_n\|, \quad n=1,2,\ldots,\infty,
\end{equation}
which, by \eqref{betan}, turns into the bound \eqref{kb3s} after
taking the limit $n\to\infty$. Solving inequality~\eqref{kb3s} for
$\kappa$ and recalling that $\kappa=\|K\|$, we conclude the estimate
\eqref{KZconv}.
\end{proof}

\begin{remark}
\label{REx4} The bound \eqref{KZconv} is sharp. In fact, equality
\eqref{sharp3} in Example \ref{Ex1} shows that, for the spectral
dispositions \eqref{Z01conv} where $\spec(Z_0)$ lies in a finite gap
of $\spec(Z_1)$, equality prevails in  \eqref{KZconv}.
\end{remark}

The strongest a posteriori bound for the solution $K$ is
obtained under the assumptions that the spectra of $Z_0$ and $Z_1$
are subordinated, i.e.
\begin{equation}
\label{Z0Z1subord} \max\,\spec(Z_0)<\min\,\spec(Z_1) \ \text{\, or \
} \max\,\spec(Z_1)<\min\,\spec(Z_0),
\end{equation}
and that $A_0$ and
$A_1$ are bounded.

\begin{theorem}
\label{Ttan2T} Assume that $L=A+V$ satisfies Assumption \ref{Hypo1}
and suppose, in addition, that the operators $A_0$, $A_1$ (and hence
$A$ and $L$) are bounded. Let the Riccati equation \eqref{RicABB}
have a weak (and hence strong) solution $K\in\cB(\fH_1,\fH_0)$ such
that $\|K\|<1$.  If the spectra of the operators $Z_0=A_0+BK$ and
$Z_1=A_1-B^*K^*$ are subordinated, that is,
\begin{equation}
\label{Z01subord} \max\,\spec(Z_0)<\min\,\spec(Z_1) \ \text{\, or \
} \max\,\spec(Z_1)<\min\,\spec(Z_0),
\end{equation}
then
\begin{equation}
\label{AK1ds}
\|K\|\leq\tan\left(\frac{1}{2}\arctan\frac{2\|B\|}{\,\delta_{Z_0,Z_1}}\right).
\end{equation}
\end{theorem}

\begin{proof}
As in the proof of Theorem \ref{Ttans}, we apply Lemma \ref{LRKK}
and rewrite the Riccati equation in the form \eqref{Ric-t2t};  note
that the self-adjoint operators $\Lambda_0$ and $\Lambda_1$ on the
left-hand side of \eqref{Ric-t2t} are bounded by Theorem \ref{Lss}
since $A_0$ and $A_1$ are bounded, $B$ is bounded, and $K$ is a
uniform contraction. Assuming that $B\neq 0$, we again have $K\ne 0$
(cf.\ \eqref{Kne0}).

Let $U$ be the partial isometry in the polar decomposition $K=U|K|$
of $K$ (see the proof of Theorem \ref{Ttans}). By Lemma \ref{phiT}
with $\varphi(z)=(1-z)^{1/2}$, $M=U^*$, and $N=|K|^2U^*$, we~obtain
\begin{align}
\nonumber
U^*(I-KK^*)^{1/2}&=U^*(I-U|K|^2U^*)^{1/2}\\
\nonumber
&=(I-U^*U|K|^2)^{1/2}U^*\\
\nonumber
                 &=(I-|K|^2)^{1/2}U^*.
\end{align}
Here, in the last step, we have used the property that
$U$ is an isometry on $\Ran(|K|)=\Ran(K^*)$ by \eqref{Kisom} so that
$U^*U|K|=|K|$ and $U^*U|K|^2=|K|^2$.

If we apply the operator $U^*$ to both sides of
\eqref{Ric-t2t} from the left, we arrive at an equation that only
involves $|K|$, but not $K$ and $K^*$ themselves:
\begin{equation}
\label{Ric-MK}
|K|\Lambda_0-U^*\Lambda_1U\,|K|=-(I-|K|^2)^{1/2}U^*B^*(I-|K|^2)^{1/2}.
\end{equation}

Let $\kappa=\max\,\spec(|K|)$. Clearly, $0<\kappa=\|K\|<1$. If
$\kappa$ is an eigenvalue of $|K|$ and $x\in \fH_0$, $\|x\|=1$, is
an eigenvector of $|K|$ at $\kappa$, that is, $|K|x=\kappa x$, then
\eqref{Ric-MK} immediately implies that
\begin{equation}
\label{Ric-kappa}
\frac{\kappa}{1-\kappa^2}\bigl((\Lambda_0x,x)-(\Lambda_1Ux,Ux)\bigr)=-(U^*B^*x,x).
\end{equation}
Since $x\in\Ran(|K|)$, by \eqref{Kisom} we have $\|Ux\|=\|x\|=1$ so that
\begin{alignat}{2}
\label{Lam1} (\Lambda_1Ux,Ux) & \geq\min\,\spec(\Lambda_1)  & &
=\min\,\spec(Z_1),\\
\label{Lam0} (\Lambda_0x,x) & \leq\max\,\spec(\Lambda_0)& &
=\max\,\spec(Z_0).
\end{alignat}
Because the spectra of $Z_0$ and $Z_1$ are subordinated by
assumption \eqref{Z01subord}, the inequalities \eqref{Lam1} and
\eqref{Lam0} yield that
$$
\bigl|(\Lambda_0x,x)-(\Lambda_1Ux,Ux)\bigr|\geq\delta_{Z_0,Z_1}>0.
$$
This and \eqref{Ric-kappa} imply the inequality
\begin{equation}
\label{kappa-bound} \frac{\kappa}{1-\kappa^2}\leq
\frac{\|B\|}{\,\delta_{Z_0,Z_1}}.
\end{equation}
If $\kappa$ is not an eigenvalue of $|K|$, it belongs to the
essential spectrum of $|K|$.
We introduce a singular sequence $\{x_n\}_{n=1}^\infty$ of $|K|$ at $\kappa$ as in
\eqref{xn}; in particular, $\|x_n\|=1$.
With this choice of~$x_n$, \eqref{Ric-MK} implies that
\begin{equation}
\label{Ric-kappan}
\frac{\kappa}{1-\kappa^2}\bigl((\Lambda_0x_n,x_n)-(\Lambda_1Ux_n,Ux_n)\bigr)
=-(U^*B^*x_n,x_n)+\alpha_n
\end{equation}
where
$$
\alpha_n=\frac{(U^*\Lambda_1U\varepsilon_n,x_n)-
(\Lambda_0x_n,\varepsilon_n)-(U^*B^*\zeta_n,\zeta_n)}{1-\kappa^2}-
\frac{(U^*B^*\zeta_n,x_n)+(U^*B^*x_n,\zeta_n)}{\sqrt{1-\kappa^2}}.
$$
Because of $x_n\in\Ran(|K|)$ and \eqref{Kisom}, we have
$\|Ux_n\|=\|x_n\|=1$. Now the same reasoning as in \eqref{Lam1} and
\eqref{Lam0} yields that
$$
\bigl|(\Lambda_0x_n,x_n)-(\Lambda_1Ux_n,Ux_n)\bigr|\geq\delta_{Z_0,Z_1}>0.
$$
Hence \eqref{Ric-kappan} shows that
\begin{equation}
\label{kappan-bound} \frac{\kappa}{1-\kappa^2}\leq
\frac{\|B\|}{\delta_{Z_0,Z_1}}+\frac{|\alpha_n|}{\delta_{Z_0,Z_1}},
\quad n=1,2,\ldots,\infty.
\end{equation}
As $\|x_n\|=1$ and both $\Lambda_0$ and $\Lambda_1$ are bounded
operators, \eqref{epsn} and \eqref{zetn} show that $\alpha_n\to 0$
for $n\to\infty$. Taking the limit $n\to\infty$ in
\eqref{kappan-bound}, we again arrive at inequality
\eqref{kappa-bound}.

To complete the proof it remains to notice that, by the formula for
double arguments of the tangent function,  the left-hand side of
\eqref{kappa-bound} may be written as
\begin{equation}
\label{arctan} \frac{\kappa}{1-\kappa^2} = \frac{1}{2}
\tan(2\arctan\kappa),
\end{equation}
and to recall that $\kappa=\|K\|$.
\end{proof}

\begin{remark}
\label{Rt2t-sharp} The bound \eqref{AK1ds} is optimal. This may be
seen from Example \ref{Ex2} where $\fH_0=\fH_1=\bbC$ and $A_0=-d/2$,
$A_1=d/2$, $B=b$ with $b,d\in\bbR$, $0< b<d/2$. In fact, by
\eqref{arctan}, equality \eqref{Ksharp-t2t} therein is equivalent to
$$
\|K\|=\tan\left(\frac{1}{2}\arctan\frac{2\|B\|}{\delta_{Z_0,Z_1}}\right).
$$
\end{remark}

\section{Proofs of Theorems \ref{ThTheta} and \ref{ThTgen}}
\label{SecMain}

Using the results of Section \ref{SecKbounds}, we are now able to
prove our main results, Theorems \ref{ThTheta} and Theorem
\ref{ThTgen}, which were formulated in the introduction. In
particular, Theorem \ref{ThTheta} appears to be a corollary to
Theorem \ref{Tmain} in Section \ref{SecKbounds}.
\medskip

\noindent\textit{Proof of Theorem \ref{ThTheta}}. According to the
definitions \eqref{sigs}, \eqref{sig'} and Theorem \ref{Lss}, we
have
\[
   \sigma_i= \spec(A_i), \quad \sigma_i'= \spec(Z_i), \qquad i=0,1,
\]
with $Z_0=A_0+BK$,  $\dom(Z_0)=\dom(A_0)$ and $Z_1=A_1-B^*K^*$,
$\dom(Z_1)=\dom(A_1)$ as above.

We prove the theorem in the case
$\dist(\sigma'_0,\sigma_1)=\delta_{Z_0,A_1}>0$; the case
$\dist(\sigma'_1,\sigma_0)>0$ may be reduced to the first case by
replacing the involution $J$ with $J'=-J$ and making the
corresponding index changes in the notations of Assumption
\ref{Hypo1}.

By assumption, $\fH'_0$ is a maximal uniformly positive subspace of
the Krein space $\fK=\{\fH,J\}$. Thus Lemma \ref{Lmaxu} implies that
$\fH'_0$  is the graph of a uniform contraction $K:\,\fH_0\to\fH_1$,
i.e.\ $\fH'_0=\cG(K)$. By assumption, $\fH'_0$ is also a reducing
and hence invariant subspace of $L$. Now Lemma \ref{Lgraph} yields
that $K$ is a strong solution to the operator Riccati equation
\eqref{RicABB}. By Lemma \ref{Lspn} and Theorem \ref{Lss}, we know
that $\fH'_1= \cG(K^*)  = \cG(K)^{[\perp]} = {\fH'_0}^{[\perp]}$ and
hence, by \eqref{KThet} and definition \eqref{ThetaDef},
\begin{equation}
\label{TKKT} \|\tan\Theta_0\|= \Theta(\fH_0,\fH'_0) = \|K\|=\|K^*\|=
\Theta(\fH_1,\fH'_1)= \|\tan\Theta_1\|.
\end{equation}
Now both claims (i) and (ii) are immediate consequences of the
respective statements (i) and (ii) of Theorem \ref{Tmain}. \qed

\medskip

\begin{remark}
\label{XXYYY1} Under the natural assumption that
$\|V\|<\delta_i$ for some $i\in\{0,1\}$, the bound \eqref{TanT0s} is
optimal with respect to the mutual positions of the sets $\sigma_i$
and $\sigma'_{1-i}$ described in condition \eqref{disp1}. This
follows from Remark \ref{REx1} and the subsequent examples together
with the equalities~\eqref{TKKT}.
\end{remark}

Theorems \ref{Ttanp}, \ref{Ttans}, and \ref{Ttan2T}  enable us to
prove Theorem \ref{ThTgen}.
\medskip

\noindent\textit{Proof of Theorem \ref{ThTgen}}. By definition
\eqref{sig'} and Theorem \ref{Lss}, we have
\[
   \sigma_i'= \spec(Z_i), \qquad i=0,1,
\]
with $Z_0=A_0+BK$,  $\dom(Z_0)=\dom(A_0)$ and $Z_1=A_1-B^*K^*$,
$\dom(Z_1)=\dom(A_1)$. Hence assumption \eqref{dist'} implies that
$\delta_{Z_0,Z_1} = \dist(\sigma'_0,\sigma'_1) = \widehat \delta >
0$ by \eqref{delZZ}.

As in the proof of Theorem~\ref{ThTheta}, we conclude that $\fH'_0$
is the graph $\cG(K)$ of a uniformly contractive strong solution $K$
to the operator Riccati equation \eqref{RicABB}, while $\fH'_1$ is
the graph $\cG(K^*)$ of the adjoint of $K$.

For claim (i), the bound \eqref{T2TN} follows from estimate
\eqref{AKpid} in Theorem \ref{Ttanp} using relation \eqref{TKKT}.

For claim (ii), the bound \eqref{T2TNs} for $i=0$ follows from
estimate \eqref{KZconv} in Theorem \ref{Ttans}, again using relation
\eqref{TKKT}; for $i=1$ it follows from the case $i=0$ by passing
from $J$ to the new involution $J'=-J$.

For claim (iii), the bound \eqref{T2T} follows from estimate
\eqref{AK1ds} in Theorem~\ref{Ttan2T} if we use \eqref{TKKT} and the
facts that $\|\tan\Theta_j\|=\tan\|\Theta_j\|$ and, by
\eqref{Thetapi4}, $\|\tan2\Theta_j\|=\tan2\|\Theta_j\|$, $j=0,1$.
\qed

\medskip

\begin{remark}
\label{XXYYY2} The bound \eqref{T2TNs} is sharp whenever the gap of
$\sigma'_{1-i}$ containing $\sigma'_i$ is finite. For $i=0$ this
follows from the fact that the estimate \eqref{KZconv} in Theorem
\ref{Ttans} is sharp by Example \ref{Ex1} (see
Remark~\ref{REx4}) together with the identity  \eqref{TKKT}; for
$i=1$ it follows using the involution $J'=-J$ instead of $J$.
\end{remark}

\begin{remark}
\label{XXYYY3} The bound \eqref{T2T} is best possible. This follows
from the fact that the estimate \eqref{AK1ds} in
Theorem~\ref{Ttan2T} is sharp by Example \ref{Ex2} (see Remark
\ref{Rt2t-sharp}) together with \eqref{Thetapi4} and \eqref{TKKT}.
\end{remark}

\section{Estimates for the perturbed spectra}
\label{SecSpec}

In the next section we want to use the operator angle bounds of
Theorems \ref{ThTheta} and \ref{ThTgen} to prove a priori bounds on
the variation of spectral subspaces of the self-adjoint operator $A$
under an off-diagonal $J$-self-adjoint perturbation $V$. To this
end, we establish some tight enclosures for the spectral components
of the perturbed operator $L=A+V$ in the present section.

We assume that the initial spectra $\sigma_0=\spec(A_0)$ and
$\sigma_1=\spec(A_1)$ of the block diagonal entries $A_0$ and $A_1$
of $A$ (see \eqref{Adiag}) are disjoint, i.e.\
\begin{equation}
\label{ddist} d=\dist(\sigma_0,\sigma_1)>0.
\end{equation}
Then the subspaces $\fH_0$ and $\fH_1$ introduced in Assumption
\ref{Hypo1} are the spectral subspaces of $A$ associated with the
spectral components $\sigma_0$ and $\sigma_1$, respectively.

In the following, our aim is to find certain bounds for the
perturbation $V$ and a constant $r_{V}\ge 0$ such that
\begin{equation}
\label{Ors} \dist(\sigma_0',\sigma_1')>0 \quad\text{and} \quad
\sigma'_i\subset O_{r_{V}}(\sigma_i) \cap \bbR, \ \ i=0,1.
\end{equation}
This yields the lower bounds
\begin{equation}
\label{lowbound} \delta_i= \dist(\sigma_i,\sigma'_{1-i}) \ge
d-r_{V}, \ \ i=0,1, \quad\text{and} \quad \widehat \delta =
\dist(\sigma'_0,\sigma'_1) \ge d - 2\, r_{V}.
\end{equation}
Together with the estimates in Theorems \ref{ThTheta} and
\ref{ThTgen}, respectively, they will give us the desired a priori
estimates for $\tan \Theta_j$, depending only on the initial
distance $d$ of the unperturbed spectra and on the norm of $V$.

For completely arbitrary (i.e.\ not necessarily off-diagonal)
perturbations $V$ of the self-adjoint operator $A$, it is well-known
that the assumption
\begin{equation}
\label{Vd2} \|V\|<\frac{d}{2}
\end{equation}
guarantees that \eqref{Ors} holds with
\begin{equation}
\label{dsigs} r_{V} = \|V\|
\end{equation}
(see, e.g., \cite[Section V.4]{K}) and hence, by \eqref{lowbound},
\begin{equation}
\label{dsigi1i} \delta_i\geq d- \|V\| > \frac d2, \ \ i=0,1, \quad
\widehat \delta \geq d- 2 \|V\|> 0.
\end{equation}

For off-diagonal perturbations $V$, earlier results in \cite{KMM3},
\cite{KMM4} for self-adjoint $V$ and in \cite{AlMoSh} for
non-symmetric $V$ show that the constant $r_V$ may be improved
considerably. In the following we extend these results under the
sole assumption \eqref{ddist} that the spectral components
$\sigma_0$ and $\sigma_1$ of $A$ are disjoint.

Unlike the previous sections, we do not always require $A$ to be
self-adjoint and $V$ to be $J$-self-adjoint; here we use the
following more general setting.

\begin{hypothesis}
\label{Hypo2} Let $\fH_0$ and $\fH_1$ be complementary orthogonal
subspaces of the Hilbert space~$\fH$. Assume that $A$ is a closed
operator on $\fH=\fH_0\oplus\fH_1$ diagonal with respect to this
decomposition,~i.e.
\[
A=\left(\begin{array}{cc} A_0 & 0\\ 0 & A_1\end{array}\right),\quad
\dom(A)=\dom(A_0)\oplus\dom(A_1),
\]
where $A_0$ and $A_1$ are closed operators on $\fH_0$ and $\fH_0$,
respectively. Suppose that $V$ is an off-diagonal bounded operator
on $\fH$, i.e.
\[
V=\left(\begin{array}{cc} 0 & B\\ C & 0\end{array}\right)
\]
with $B\in\cB(\fH_1,\fH_0)$, $C\in\cB(\fH_0,\fH_1)$, and let
\begin{equation}
\label{L-nonsymm} L=A+V=\left(\begin{array}{cl} A_0 & B\\ C &
A_1\end{array}\right), \quad\dom(L)=\dom(A).
\end{equation}
\end{hypothesis}

The following two elementary auxiliary results are used in the
proofs below.

\begin{lemma}
\label{schur} Assume that $L=A+V$ satisfies Assumption \ref{Hypo2}
and define the Schur complement $S_0$ of $A$ by
\[
  S_0(\la) := A_0-\la - B(A_1-\la)^{-1}C, \quad \dom\big(S_0(\la)\big) := \dom (A_0),
\]
for $\la \in \rho(A_1)$. Then
\begin{enumerate}
\item[(i)] the resolvent set $\rho(S_0) :=
\big\{ \la\in \rho(A_1) : S_0(\la) \text{ is bijective} \big\}$ of
$\,S_0$ satisfies
 \[ \rho(S_0) = \rho(L)\cap \rho(A_1); \]
 \item[(ii)]
 for $\la\in \rho(A_0)\cap \rho(A_1)$ we have
 \begin{equation}
 \label{neumann}
   \| B(A_1-\la)^{-1}C(A_0-\la)^{-1}\|<1 \implies \la\in \rho(L).
 \end{equation}
\end{enumerate}
\end{lemma}

\begin{proof}
Both claims (i) and (ii) are well known (see, e.g.,
\cite{Tretter2008}); we recall the short proofs for the convenience
of the reader.

(i) It is easy to check that, for arbitrary $f\in \fH_0$, $g\in
\fH_1$ and $x\in \dom(A_0)$, $y\in \dom(A_1)$,
\[
(L-\la) \binom xy = \binom fg \iff \left\{ \begin{array}{rl}
S_0(\la) x \!\!\!\!
&= f - B (A_1-\la)^{-1} g, \\[1mm] y\!\!\!\!
&= (A_1-\la)^{-1} (g - Cx), \end{array} \right.
\]
which proves the claim.

(ii) For $\la\in \rho(A_0)\cap \rho(A_1)$, one can write
\[
 S_0(\la) = \bigl( I - B(A_1-\la)^{-1}C(A_0-\la)^{-1} \bigr)
 \bigl(A_0-\la\bigr).
\]
Thus a Neumann series argument together with (i) proves
\eqref{neumann}.
\end{proof}

\begin{lemma}
\label{elementary} Let $a,b,v \in \bbR$ be such that $\delta:=b-a>0$
and $0\leq v<\delta/2$. Then
\[
  (t-a)(b-t) > v^2 \ \iff \ a+r< t < b-r
\]
where
\begin{align}
\label{eps}
 r = & \frac{\delta}{2}-\sqrt{\frac{\delta^2}{4}-v^2}
   = v\,\, \tan\left( \frac 12 \arcsin \frac{2v}{\delta} \right).
\end{align}
\end{lemma}

\begin{proof}
The claims are obvious; for the last equality, observe the formula
for double arguments of the sine function in terms of the tangent
function.
\end{proof}

In the following theorem, we consider the case where the diagonal
entries $A_0$ and $A_1$ of the block operator matrix
\eqref{L-nonsymm} are self-adjoint and their spectra do not
intersect; the bounded perturbation $V$ need not have any symmetry
here.

\begin{theorem}
\label{c1} Assume that $L=A+V$ satisfies Assumption \ref{Hypo2} and
let  $A_0$ and $A_1$ be self-adjoint. Assume, in addition, that
their spectra $\sigma_0=\spec(A_0)$ and $\sigma_1=\spec(A_1)$ are
disjoint,~i.e.\
$$
d=\dist(\sigma_0,\sigma_1)>0,
$$
and let the entries $B$ and $C$ of $\,V$ be such that
\begin{equation}
\label{BCd2} \sqrt{\|B\|\|C\|}<\frac{d}{2}.
\end{equation}
Then
\begin{equation}
\label{OrBC}
{\spec(L) = \sigma_0' \,\dot\cup\,
\sigma_1', \quad \sigma'_i\subset O_{r_{V}}(\sigma_i) \cap \bbR, \,\,
i=0,1,}
\end{equation}
where $r_V$ is given by
\begin{equation}
\label{r} r_V=\sqrt{\|B\|\,\|C\|} \ \tan \left( \frac{1}{2} \arcsin
\frac{2 \sqrt{\|B\|\,\|C\|}}{d} \right) < \sqrt{\|B\|\,\|C\|};
\end{equation}
in particular, if $\,V$ is $J$-self-adjoint and hence $C=-B^\ast$,
then
\begin{equation}
\label{r-symm} r_V=\|V\| \ \tan \left( \frac{1}{2} \arcsin \frac{2
\|V\|}{d} \right) < \|V\|.
\end{equation}
\end{theorem}

\begin{proof}
Throughout the proof, we assume that $\la\in \bbC$ is such that
\begin{equation}
\label{lssr} \dist (\la,\sigma_0\cup\sigma_1) > r_V;
\end{equation}
hence, in particular, $\la\in\rho(A_0)\cap\rho(A_1)$. Since $A_0$
and $A_1$ are assumed to be self-adjoint, we have
$\|(A_i-\la)^{-1}\|=1/\dist(\la,\sigma_i)$, $i=0,1$, and thus
\begin{equation}
\label{Ineq1}
   \| B(A_1-\la)^{-1}C(A_0-\la)^{-1}\|
   \le \frac{\|B\|\,\|C\|}{\dist(\la, \sigma_0) \dist(\la, \sigma_1)}.
\end{equation}
First we consider the case that $\la$ lies in a strip of the form
\begin{equation}
\label{strip} \{ z\in\bbC \ | \ a+r_V < \Real z < b-r_V \}
\end{equation}
where $(a,b)$ is a finite gap of the spectrum of $A$ with $a\in
\sigma_0$ and $b\in \sigma_1$; the case $a\in \sigma_1$ and $b\in
\sigma_0$ is analogous. Then we have $b-a \ge d$ and hence, by
assumption \eqref{BCd2} and Lemma~\ref{elementary}, we obtain
\begin{equation}
\label{Ineq2} \frac{\|B\|\,\|C\|}{\dist(\la, \sigma_0) \dist(\la,
\sigma_1)}= \frac{\|B\|\,\|C\|}{|\la-a||\la-b|}\leq
\frac{\|B\|\,\|C\|}{(\Real\la-a)(b-\Real\la)}<1.
\end{equation}
Now \eqref{Ineq1} and \eqref{Ineq2} together with
Lem\-ma~\ref{schur}\,(ii) show that $\la\in \rho(L)$.

If $\la$ does not belong to a strip of the form~\eqref{strip} with
$a\in \sigma_i$ and $b\in \sigma_{1-i}$ for $i=0$ or $i=1$, then it
is not difficult to check that, by \eqref{lssr}, either for $i=0$ or
$i=1$
\begin{equation}
\label{Ineq3} \dist(\Real\lambda,\sigma_i)\geq d-r_V, \quad
\dist(\lambda,\sigma_{1-i})>r_V.
\end{equation}
Combining \eqref{Ineq1}, \eqref{Ineq3} together with the definition
\eqref{r} of $r_V$, we arrive at the estimate
\[
   \| B(A_1-\la)^{-1}C(A_0-\la)^{-1}\| < \frac{\|B\|\,\|C\|}{r_V(d-r_V)} = 1.
\]
Hence Lemma \ref{schur}\,(ii) again shows that $\la\in \rho(L)$.
\end{proof}

\begin{remark}
\label{compare} Theorem \ref{c1} improves the spectral bounds given
in \cite[Remarks~4.6 and~4.13]{AlMoSh}. In fact, the bound $r$ in
\cite[(4.11)]{AlMoSh} can be written equivalently as
\begin{equation}
\label{rtans} r \!=\! \sqrt{\|B\|\|C\|} \tanh \!\left( \!\frac 12
\arctanh \frac {2 \sqrt{\|B\|\|C\|}}\delta \right) \!=\!
\sqrt{\|B\|\|C\|} \tan \!\left( \!\frac 12 \arcsin \frac {2
\sqrt{\|B\|\|C\|}}\delta \right)
\end{equation}
and applies whenever $\|B\|\|C\| < \delta/2$. Here, by
\cite[Theorem~4.11]{AlMoSh}, we have $\delta=2d/\pi<d$ in the general
case \eqref{ddist} and $\delta=d$ if one additionally assumes that
\begin{equation}
\label{conv} \conv(\sigma_0)\cap\sigma_{1}=\emptyset \quad \text{or}
\quad \sigma_0\cap\conv(\sigma_1)=\emptyset.
\end{equation}
This shows that, under the additional assumption \eqref{conv}, the
spectral bound in Theorem \ref{c1} coincides with the one in
\cite[Remarks~4.6]{AlMoSh}, whereas in the general case
\eqref{ddist} Theorem \ref{c1} holds with the weaker norm bound
$\|B\|\|C\| < d/2$ on $V$ and the bound $r_V$ is strictly smaller
than the bound $r$ in \cite[(4.11)]{AlMoSh}.
\end{remark}

\begin{remark}
The spectral bound \eqref{OrBC} with $r_V$ given by \eqref{r} is
optimal. This may be seen from \cite[Exam\-p\-les~4.15,
4.16]{AlMoSh}.
\end{remark}

In the next two theorems we drop the assumption that $A$ is
self-adjoint. Instead we impose conditions to ensure that the
components $A_0$ and $A_1$ of $A$ satisfy certain resolvent
estimates.

To this end, we use the \emph{numerical range} $W(T)$ of a linear
operator $T$ with domain $\Dom(T)$ in a Hilbert space, defined as
\[
 W(T)=\{ (Tx,x) \ | \ x\in \Dom(T), \, \|x\|=1 \}.
\]
Recall that the numerical range is always convex and that $\spec(T)
\subset \overline{W(T)}$ if every (of the at most two)
connected components of $\bbC\setminus W(T)$ contains at least
one point of $\rho(T)$ (see \cite[Theorems~V.3.1 and V.3.2]{K}); in
this case,
\[
 \| (T-\la)^{-1} \| \le \frac 1{\dist(\la, W(T))}, \quad \la \not\in
\overline{W(T)}.
\]

First we consider the case where the spectra and the numerical
ranges of $A_0$ and $A_1$ are separated by a vertical strip.

\begin{theorem}
\label{c2} Assume that $L=A+V$ satisfies Assumption \ref{Hypo2} and
that there exist $a, b \in \bbR$ such that $d=b-a>0$,
\begin{equation}
 \label{ch1} \Re W(A_0) \le a < b \le \Re W(A_1) \quad \text{or} \quad
 \Re W(A_1) \le a < b \le \Re W(A_0),
\end{equation}
and
\begin{equation}
 \label{ch2}
  \{ z \in \bbC \,\,\bigl|\,\, a < \Re z < b \} \,\cap\, \rho(A_0) \,\cap\,
\rho(A_1) \not = \emptyset.
\end{equation}
If
$$\sqrt{\|B\|\,\|C\|} < d/2,$$
and $r_V$ is defined as in \eqref{r}, then
\[
  \big\{ z\in \bbC \,\,\bigl|\,\, a + r_V < \Re z <  b - r_V \big\}
\subset \rho(L).
\]
\end{theorem}

\begin{proof}
Without loss of generality, we suppose that $\Re W(A_0) \le a < b
\le \Re W(A_1)$. In this case the assumptions \eqref{ch1},
\eqref{ch2} imply that (see \cite[Theorem~V.3.2]{K})
\begin{equation}
\label{ssAA}
 \spec(A_0) \subset \overline{W(A_0)} \subset \{ z \in \bbC
 \,\,\bigl|\,\,\Re z \le a\}, \quad \spec(A_1) \subset
 \overline{W(A_1)} \subset \{ z \in \bbC \,\,\bigl|\,\, \Re z \geq
 b\}.
\end{equation}

Let $\la\in\bbC$ be such that $a+r_V < \Re \la <  b+r_V$ and hence
$\la\in \rho(A_0)\cap \rho(A_1)$. Then, by \eqref{ssAA},
\[
 \| (A_0 - \la)^{-1} \| \le \frac 1{\dist\big(\la,W(A_0)\big)} \le
 \frac 1{\Re \la - a}, \quad \| (A_1 - \la)^{-1} \|  \le \frac
 1{\dist\big(\la,W(A_1)\big)} \le \frac 1 {b - \Re \la}.
\]
Thus Lemma~\ref{elementary} shows that
\[
  \| B(A_1-\la)^{-1}C(A_0-\la)^{-1}\| \le
  \frac{\|B\|\,\|C\|}{(\Re \la- a)(b-\Re \la)} <1,
\]
and hence $\la\in \rho(L)$ by Lemma~\ref{schur}\,(ii).
\end{proof}

Next we consider the case where the spectra and the numerical ranges
of $A_0$ and $A_1$ (and hence of $A$) lie in one half-plane and the
perturbation $V$ is $J$-self-adjoint, i.e.\ $C=-B^\ast$.

While all previous theorems were of perturbational character, the
following theorem is not. In fact, we prove implications of the form
$$
\Re \spec(A) \le \Re W(A) \le a \ \implies \ \Re \spec(L) \le a,
$$
independently of the norm of $V$.

This type of results relies on the \emph{quadratic numerical range}
$W^2(L)$ of the operator $L$ with respect to the block
representation \eqref{L-nonsymm}. The set $W^2(L)$ is defined as
(see \cite[(2.2)]{LT1998} and also
\cite[Definition~3.1]{Tretter2009} or
\cite[Definition~2.5.1]{Tretter2008})
\[
 W^2(L) = \bigcup_{\efrac{x\in \dom(A_0),\,\, y\in
\dom(A_1)}{\|x\|=\|y\|=1}} \spec (L_{x,y})
\]
where $L_{x,y} \in M_2(\bbC)$ is a $2\times 2$ matrix given by
\[
 L_{x,y} := \left(\begin{array}{cc}
 (A_0x,x) & (By,x)\\
 (Cx,y) & (A_1y,y)
 \end{array}\right), \quad x\in \dom(A_0), \,  y\in
 \dom(A_1),\, \|x\|=\|y\|=1.
\]

The quadratic numerical range is not convex and may consist of at
most two connected components. It is always contained in the
numerical range, $W^2(L) \subset W(L)$, and the inclusion
$\spec(L)\subset \overline{W^2(L)}$ holds if every connected
component of $\bbC\setminus W^2(L)$ contains at least one point of
$\rho(L)$ (see \cite[Theorem~2.1]{LT1998}, \cite[Proposition~3.2,
Theorems~4.2 and 4.7]{Tretter2009} or \cite[Theorem~2.5.3,
Theorems~2.5.10 and~2.5.15]{Tretter2008})).

\begin{theorem}
\label{qnr2} Assume that $L=A+V$ satisfies Assumption \ref{Hypo2}
with $C=-B^\ast$ and let $a,b\in\bbR$.
\begin{enumerate}
\item [(i)] If\,
$\Re W(A) \le a$\, and  \,$\rho(A)\cap\big\{ z\in\bbC \,\,\bigl|\,\,
\Re z > a \big\} \not = \emptyset$, then\, $\Re \spec(L) \le a$.
\item [(ii)] If\,
$\Re W(A) \ge b$\, and \,$\rho(A)\cap\big\{ z\in\bbC \,\,\bigl|\,\,
\Re z < b \big\} \not = \emptyset$, then \,$\Re \spec(L) \ge b$;
\end{enumerate}
in particular, if $A$ is self-adjoint, we have
\begin{equation*}
\label{Ass} \inf\,\spec(A)\leq\Real\,\spec(L)\leq\sup\,\spec(A).
\end{equation*}
\end{theorem}

\begin{proof}
We prove (i); the proof of (ii) is completely analogous. By the
assumption on $A$ and since $V$ is bounded, it is obvious that
$\big\{ z\in\bbC \,\,\bigl|\,\, \Re z >a  \big\} \cap \rho(L) \not =
\emptyset$ for $L=A+V$. Hence $\spec(L) \subset W^2(L)$ and so it
suffices to show that $\Re W^2(L) \le a$.

Suppose, to the contrary, that there exists a $\la\in W^2(L)$ with
$\Re \la > a$. By the definition of $W^2(L)$, there are $x\in
\dom(A_0)$, $y\in \dom(A_1)$, $\|x\|=\|y\|=1$, such that $\la\in
\spec (L_{x,y})$, i.e.
\[
 0=\det \big( L_{x,y} - \la \big) = \big( (A_0x,x) - \la \big) \big(
(A_1y,y) - \la \big) + |(By,x)|^2,
\]
where we have used that $C=-B^\ast$. Splitting into real and
imaginary parts, we conclude that
\begin{align*}
 0 & = \Re \big( (A_0x,x) - \la \big) \,\Im \big( (A_1y,y) -
\la \big) + \Im \big( (A_0x,x) - \la \big)
\,\Re \big( (A_1y,y) - \la \big),\\
-|(By,x)|^2 &= \Re \big( (A_0x,x) - \la \big) \,\Re \big( (A_1y,y) -
\la \big) - \Im \big( (A_0x,x) - \la \big) \, \Im \big( (A_1y,y) -
\la \big).
\end{align*}
Solving the first equation for $\Im \big( (A_1y,y) - \la \big)$ and
inserting into the second equation, we find
\begin{equation}
\label{contra}
 -|(By,x)|^2 = \Big( \big( \Re (A_0x,x) - \Re \la \big)^2 +  \big(
\Im (A_0x,x) - \Im \la \big)^2 \Big) \frac{\Re (A_1y,y) - \Re
\la}{\Re (A_0x,x) - \Re \la}.
\end{equation}
By the assumption on $A$, we have $\Re \la > a \ge \Re W(A) = \Re
\big( W(A_0) \cup W(A_1) \big)$ and hence both the first and the
second factor on the right hand side of \eqref{contra} are positive,
a contradiction.
\end{proof}

\section{A priori bounds on variation of spectral subspaces}
\label{SecBapr}

In this section we use the (semi-) a posteriori norm bounds for
the operator angles from Theorems \ref{ThTheta} and \ref{ThTgen}
together with the spectral estimates from Section \ref{SecSpec} to
derive a priori estimates for the variation of the spectral
subspaces of the self-adjoint operator $A$ under a $J$-self-adjoint
off-diagonal perturbation $V$.

To ensure that solutions of the corresponding Riccati equations
exist, we use some results of \cite{AlMoSh} and \cite{Veselic2}.
They provide sufficient conditions on the perturbation $V$ and
spectral sets $\sigma_0$ and $\sigma_1$ guaranteeing that the
perturbed operator $L=A+V$ is similar to a self-adjoint operator on
a Hilbert space and that the spectral subspaces of $L$ associated
with the perturbed spectral sets $\sigma'_0$  and $\sigma'_1$,
both being real in this case, are maximal uniformly definite in the
Krein space $\fK$.

In order to formulate the conditions from \cite{Veselic2}, we need
to specify all those finite gaps of the spectrum of $A$
that separate the subsets $\sigma_0$ and $\sigma_1$. We denote
these gaps by $\Delta_n$, where $n\in\bbZ$ runs from $N_-$ through
$N_+$ with $-\infty\leq N_-\leq 0$, $0\leq N_+\leq +\infty$, and let
\begin{equation}
\label{Deltaj}
\Delta_n=(a_n,b_n), \quad -\infty<\ldots\leq a_{n-1}<
b_{n-1}\leq a_n<b_n\leq a_{n+1}<b_{n+1}\leq\ldots<\infty\,,
\end{equation}
assigning the value of $n=0$, say, to the gap that is closest to the
origin $z=0$. For every such gap, we have
\,$\Delta_n\cap\sigma_0=\emptyset$,\,
$\Delta_n\cap\sigma_1=\emptyset$,\, and \,$a_n\in\sigma_i$,\,
$b_n\in\sigma_{1-i}$\, where either $i=0$ or $i=1$. If the total
number $N$ of the gaps between $\sigma_0$ and $\sigma_1$ is finite,
then both $N_-$, $N_+$ are finite and $N=|N_-|+N_++1$. Otherwise, at
least one of $N_-$ and $N_+$ is infinite.

The next theorem is an immediate consequence of
\cite[Theorem~5.8]{AlMoSh} and \cite[Theorem~3 and
Corollary~4]{Veselic2}, combined with Theorem \ref{c1}.

\begin{theorem}
\label{Tsuff} Assume that $L=A+V$ satisfies Assumption \ref{Hypo1}
and let the spectra $\sigma_0=\spec(A_0)$ and $\sigma_1=\spec(A_1)$
be disjoint, i.e.\
$$d=\dist(\sigma_0,\sigma_1)>0.$$
Assume, in addition, that one of the following holds:

\begin{enumerate}
\item[(i)]
\ $\|V\|<\dfrac{d}{\pi}$\,.
\smallskip

\item[(ii)]  %
\ $\|V\|<\dfrac{d}{2}$ \, and \,
$\displaystyle\sum_{n=N_-}^{N_+} \frac{1}{b_n-a_n}<\infty$.
\end{enumerate}
Then
\[
 \spec(L) \subset \sigma_0' \,\dot\cup\, \sigma_1', \quad \sigma_i'
\subset O_{r_V}(\sigma_i) \cap \bbR, \ \ i=0,1,
\]
with $r_V$ given by \eqref{rtans}. The operator $L$ is similar to a
self-adjoint operator on $\fH$. The spectral subspaces $\fH'_0$ and
$\fH'_1$ of $\,L$ associated with the sets $\sigma'_0$ and
$\sigma'_1$ are mutually orthogonal in the Krein space $\{\fH,J\}$
and maximal uniformly positive resp.\ negative therein.
\end{theorem}

\begin{remark}
In case (i) Theorem \ref{Tsuff} follows from \cite[Theorem 5.8\,(i)
and Remark~5.11]{AlMoSh}, together with Theorem \ref{c1}. In case
(ii)
the spectral projections $\sE_L(\sigma'_i)$ of $L$
associated with its isolated spectral components $\sigma'_i$,
$i=0,1,$ are well-defined (see \cite[Co\-rol\-la\-ry~4]{Veselic2}),
the operator $J\bigl(\sE_L(\sigma'_0)-\sE_L(\sigma'_1)\bigr)$ is
self-adjoint, and there exists a $\gamma\in(0,1]$ such that
\begin{equation}
\label{JELs}
J\bigl(\sE_L(\sigma'_0)-\sE_L(\sigma'_1)\bigr)\geq\gamma I.
\end{equation}
The latter was established (under much more general assumptions on
$V$ than \eqref{Vd2}) in the proof of \cite[The\-o\-rems~1
and~3]{Veselic2}. Using inequality \eqref{JELs}, one easily verifies
that the spectral subspaces $\fH'_0=\Ran \sE_L(\sigma'_0)$  and
$\fH'_1=\Ran \sE_L(\sigma'_1)$ are maximal uniformly positive and
maximal uniformly negative, respectively. In fact, it suffices to
show the uniform definiteness of one spectral subspace (see
Corollary \ref{LCaux}).
\end{remark}

\begin{remark}
The lengths of the gaps $(a_n,b_n)$ of $\spec(A)$ separating the
sets $\sigma_0$ and $\sigma_1$ have to be uniformly bounded from
below. Apart from this, condition (i) imposes no further restriction
on the behaviour of the lengths, whereas condition (ii) requires that
$b_n-a_n$ tends to $\infty$ faster than $|n|$ as $|n|\to\infty$.
\end{remark}

The following a priori bound on the operator angles
$\Theta(\fH_i,\fH'_i)$ between the unperturbed and the perturbed
spectral subspaces $\fH_i$ of $A$ and $\fH'_i$ of $L=A+V$ improves
the corresponding bound derived in \cite[Theorem~5.8 (i)]{AlMoSh}
(see Remark~\ref{Rgimp} below).

\begin{theorem}
\label{Tbapr-gen} Suppose that $L=A+V$ satisfies the assumptions of
Theorem \ref{Tsuff}. Let $\fH_i$, $\fH_i'$ be the spectral subspaces
of $A$ corresponding to $\sigma_i$ and of $\,L$ corresponding to
$\sigma_i'$, respectively, $i=0,1$. Then the operator angles
$\Theta_j=\Theta(\fH_i,\fH'_i)$, $i=0,1,$ satisfy the estimate
\begin{equation}
\label{Tba-gen} \tan\Theta_i \leq
\dfrac{\pi}{2}\,\tan\left(\frac{1}{2}\,\arcsin\frac{2\|V\|}{d}\right),
\quad i=0,1;
\end{equation}
if, in addition,
\begin{equation}
 \label{conv2}
\conv(\sigma_0)\cap\sigma_1=\emptyset \quad \text{or} \quad
\sigma_0\cap\conv(\sigma_1)=\emptyset,
\end{equation}
then
\begin{equation}
\label{Tba-gen1} \tan\Theta_i \leq
\tan\left(\frac{1}{2}\,\arcsin\frac{2\|V\|}{d}\right), \quad i=0,1.
\end{equation}
\end{theorem}

\begin{proof}
By Theorem \ref{Tsuff}, the perturbed operator $L$ is similar to a
self-adjoint operator and its disjoint spectral components
$\sigma'_i\subset\bbR$, $i=0,1$, satisfy the inclusions \eqref{Ors}
with $r_V$ given by \eqref{rtans}. The latter implies that, for
$i=0,1$,
\[
\delta_i=\dist(\sigma_i,\sigma'_{1-i})\geq
d-r_V=\frac{d}{2}+\sqrt{\frac{d^2}{4}-\|V\|^2}
\]
and hence, by \eqref{eps},
\begin{equation}
\label{deli} \frac{\|V\|}{\delta_i} \le \frac{\|V\|}
{\frac{d}{2}+\sqrt{\frac{d^2}{4}-\|V\|^2}} = \tan\left(
\frac 12 \arcsin \frac{2\|V\|}{d} \right).
\end{equation}
In addition, Theorem \ref{Tsuff} also shows that the spectral
subspaces $\fH'_i$ associated with $\sigma_i'$ are maximal uniformly
definite. Thus all assumptions of The\-o\-rem~\ref{ThTheta}~(i)
(\ref{ThTheta}~(ii), respectively) are satisfied and the
claimed bound \eqref{Tba-gen} (\eqref{Tba-gen1}, respectively)
follows from \eqref{TanT0} (\eqref{TanT0s}, respectively) together
with \eqref{deli}.
\end{proof}

\begin{remark}
\label{opang} The operators $\tan\Theta_i$, $i=0,1$, in Theorem
\ref{Tbapr-gen} coincide with the moduli $|K|$ or $|K^\ast|$ of
uniformly contractive solutions to the Riccati equations
\eqref{RicABB} and \eqref{RicABB1}, respectively (see
Remark~\ref{Rangular} and Lemma~\ref{Lspn}); hence we always have
$\tan\Theta_i<1$.

If $L=A+V$ satisfies condition (i) in Theorem \ref{Tsuff}, i.e.\
$\|V\|<d/\pi$, then the bound \eqref{Tba-gen} is always less than
$1$; if $L$ satisfies condition (ii) in this theorem and hence
$\|V\|<d/2$, then for the bound \eqref{Tba-gen} to be less than $1$
we need to have $d$ and $V$ such that \,
$$ \|V\| < \frac d2 \sin \left( 2 \arctan \frac 2\pi \right) =
\frac d2 \frac{4\pi}{4+\pi^2} \approx \frac d2 \,
\times \,0.9060367012.$$
\end{remark}

\begin{remark}
\label{Rgimp} The bound \eqref{Tba-gen} is stronger than the
previously known bound
\begin{equation}
\label{TbAMSh} \tan\Theta_i\leq
\tanh\left(\frac{1}{2}\arctanh\dfrac{\pi\|V\|}{d}\right) =
\tan\left(\frac{1}{2}\arcsin\dfrac{\pi\|V\|}{d}\right),\quad i=0,1,
\end{equation}
from \cite[Theorem 5.8\,(i)]{AlMoSh} and extends it to perturbations
$V$ that do not satisfy the condition $\|V\|<d/\pi$ required
therein. The former is a consequence of the trigonometric inequality
$\frac{\pi}{2}\,\tan\left(\frac{1}{2}\,\arcsin(2t)\right)<
\tanh\left(\frac{1}{2}\arctanh(\pi\,t)\right)$,\,  $t\in(0,1/\pi)$.
Note that, if $\|V\|<{d}/{\pi}$, then the bound \eqref{TbAMSh} may
be written equivalently as
\begin{equation}
\label{TbAMSh1} \sin2\Theta_i\leq \frac{\pi\|V\|}{d}, \quad i=0,1.
\end{equation}

{For the particular case \eqref{conv2}, the bound \eqref{Tba-gen1}
coincides with the previously known bound \eqref{Sin2Tin} from
\cite[Theorem 5.8\,(ii)]{AlMoSh} since then the corresponding
spectral bounds $r_V$ and $r$ (defined in \eqref{r-symm} and
\eqref{rtans}, respectively) coincide (see Remark~\ref{compare}).}
\end{remark}

\begin{remark}
\label{Rlast} By \eqref{Ors} and \eqref{r-symm}, we also have the
estimate
\begin{equation}
\label{delapr} \widehat{\delta}=\dist(\sigma'_0,\sigma'_1)\geq
d-2r_V=\sqrt{d^2-4\|V\|^2}.
\end{equation}
Combining estimate \eqref{T2TN} from Theorem \ref{ThTgen}\,(i) with
inequality \eqref{delapr}, we arrive at the bound
$$
\tan\Theta_i\leq \frac \pi2 \frac{\|V\|}{\sqrt{d^2-4\|V\|^2}}, \quad
i=0,1,
$$
which is worse than \eqref{Tba-gen}, but still better than the
estimate \eqref{TbAMSh} from \cite[Theorem
5.8\,(i)]{AlMoSh}.
Combining the estimate \eqref{T2TNs} from Theorem \ref{ThTgen}\,(ii)
with the inequality \eqref{delapr} yields the~bound
$$\tan\Theta_i\leq\frac{\|V\|}{\sqrt{d^2-3\|V\|^2}}, \quad  i=0,1,$$
which is worse than the estimate
\eqref{Sin2Tin}.
\end{remark}

\begin{remark}
\label{Rlast1} If $\|V\|<d/2$ and the spectral sets $\sigma_0$ and
$\sigma_1$ are bounded and subordinated, i.e.\
$\conv(\sigma_0)\cap\conv(\sigma_1)=\emptyset$,
then combining inequality \eqref{delapr} with the estimate
\eqref{T2T} from Theorem \ref{ThTgen}\,(iii) results exactly in the
a priori sharp norm bound \eqref{Sin2T1} from Theorem \ref{IEOT}.
\end{remark}

\section{Quantum harmonic oscillator under a $\cP\cT$-symmetric perturbation}
\label{SecExHO}

In this section we apply the results of the previous sections to the
$N$-dimensional isotropic quantum harmonic oscillator under a
$\cP\cT$-symmetric perturbation.

Let $\fH=L_2(\bbR^N)$ for some $N\in\bbN$. Assuming
that the units are chosen such that $\hbar=m=\omega=1$, the
Hamiltonian of the isotropic quantum harmonic oscillator is given by
\begin{equation}
\label{Aho}
(Af)(x)=-\frac{1}{2}\Delta f(x)+\frac{1}{2}|x|^2f(x), \quad
\dom(A)=\biggl\{f\in W^2_2(\bbR^N)\,\,\biggl|\,\, \displaystyle\int_{\bbR^N}\!
dx\; |x|^4|f(x)|^2<\infty\biggr\},
\end{equation}
where $\Delta$ is the Laplacian and $W_2^2(\bbR^N)$ stands
for the Sobolev space of $L_2(\bbR^N)$-functions that have their
second partial derivatives in $L_2(\bbR^N)$.

It is well-known that the Hamiltonian $A$ is a self-adjoint operator
in $L_2(\bbR^N)$ and its spectrum consists of eigenvalues of
the form
\begin{equation}
\label{lamN}
 \la_n = n+ N/2, \quad n=0,1,2, \dots,
\end{equation}
whose multiplicities $\mu_n$ are given by the binomial coefficients
(see, e.g., \cite{LJVdJ2008} and the references therein)
\begin{equation}
\label{mun}
\mu_n=\left(\begin{array}{c}N+n-1\\n\end{array}\right), \quad n=0,1,2, \dots.
\end{equation}
For $n$ even, the corresponding eigenfunctions $f(x)$ are
symmetric with respect to space reflection $x\,\mapsto-x$
(i.e.\ $f(-x)=f(x)$). For $n$ odd, the
eigenfunctions are anti-symmetric (i.e.\
$f(-x)=-f(x)$). Hence if we partition the spectrum
$\spec(A) = \sigma_0 \,\dot\cup\, \sigma_1$ with
$$
\sigma_0=
\{n+N/2\,\,\bigl|\,\, n=0,2,4,\dots\}, \quad
\sigma_1=
\{n+N/2\,\,\bigl|\,\,n=1,3,5\ldots\},
$$
then the subspaces
\begin{equation}
\label{Hho} \fH_0=L_{2,\textrm{even}}(\bbR^N), \quad
\fH_1=L_{2,\textrm{odd}}(\bbR^N)
\end{equation}
of symmetric and anti-symmetric functions are the complementary
spectral subspaces of $A$ corresponding to the spectral components
$\sigma_0$ and $\sigma_1$, respectively. Obviously,
$$
d=\dist(\sigma_0,\sigma_1)=1.
$$

Let $\cP$ be the parity operator on $L_2(\bbR^N)$, $(\cP
f)(-x)=f(-x)$, and $\cT$ the (antilinear) operator of complex
conjugation, $(\cT f)(x)=\overline{f(x)}$, $f\in L_2(\bbR^N)$. An
operator $V$ on $L_2(\bbR^N)$ is called $\cP\cT$-symmetric if it
commutes with the product $\cP\cT$, i.e.
\begin{equation}
\label{PTsym} \cP\cT V =V\cP\cT.
\end{equation}
Clearly, the parity operator $\cP$ is a self-adjoint involution on
$L_2(\bbR^{N})$ whose spectral subspaces
$$
\Ran\sE_{\cP}(\{+1\})=L_{2,\textrm{even}}(\bbR^N), \quad
\Ran\sE_{\cP}(\{-1\})=L_{2,\textrm{odd}}(\bbR^N)
$$
coincide with the respective spectral subspaces \eqref{Hho} of the
Hamiltonian \eqref{Aho}.

{}From now on, let $V$ be the multiplication operator by a
function of the form
$$V(x)={\rm i} \,b(x), \quad x \in \bbR^N,$$
where $b\in L_\infty(\bbR^N)$ is real-valued and anti-symmetric, i.e.\
$b(x)\in\bbR$ and $b(-x)=-b(x)$ for a.e. $x\in\bbR^N$. Such an
operator $V$ is not only $\cP\cT$-symmetric on $L_2(\bbR^N)$  (see,
e.g., \cite[Section 3]{Caliceti3}) but also $J$-self-adjoint with
respect to the involution $J=\cP$. Moreover, it is anticommuting
with $\cP$ which means that such a $V$ is off-diagonal with respect
to the decomposition $\fH=\fH_0\oplus\fH_1$.

By \cite[Theorem 3.2]{Caliceti3} the spectrum of the perturbed
Hamiltonian $L=A+V$ given by
$$
(L f)(x) = -\frac{1}{2}\Delta f(x) +\frac{1}{2}|x|^2 f(x) + {\rm i} b(x) f(x),
\quad \dom(L)=\dom(A),
$$
remains real (and discrete) whenever \mbox{$\|V\|=\|b\|_{\infty}<
1/2$}. Furthermore, \cite[Proposition 3.5]{Caliceti3} implies that
for $\|V\|<1/2$ the closed $O_{\|V\|}(\lambda_n)$-neighbourhood of
the eigenvalue \eqref{lamN} of $A$ contains exactly $\mu_n$ (real)
eigenvalues $\lambda'_{n,k}$, $k=1,2,\ldots,\mu_n$, of $L$, counted
with multiplicities, where $\mu_n$ is given by \eqref{mun}.
Combining \cite[The\-o\-rem~3.2 and
Pro\-po\-si\-ti\-on~3.5]{Caliceti3} with The\-o\-rem~\ref{c1}
ensures that, in fact, the eigenvalues $\lambda'_{n,k}$ satisfy the
estimates
\[
  \big|\lambda'_{n,k} - (n+N/2)\big| < r_V,   \quad n=0,1,2,\dots,\
\ k=1,2,\ldots,\mu_n,
\]
where
\[
r_V={\|b\|_\infty \tan\left(\frac{1}{2}\arcsin(2\|b\|_\infty)\right)<\|b\|_\infty}.
\]

Further, assume that the stronger inequality
$\|V\|=\|b\|_{\infty}<1/\pi$ holds. In this case it follows from
\cite[Theorem 5.8\,(i)]{AlMoSh} that $L$ is similar to a
self-adjoint operator. At the same time, The\-o\-rem~\ref{Tbapr-gen}
implies the following bound on the variation of the spectral
subspaces \eqref{Hho}:
\begin{equation}
\label{Bosc} \tan\Theta_j\leq\frac{\pi}{2}\tan\left(\frac{1}{2}
\arcsin(2{\|b\|_\infty})\right)<1,\quad j=0,1,
\end{equation}
where $\Theta_j=\Theta(\fH_j,\fH'_j)$ denotes the operator angle
between the subspace $\fH_j$ and the spectral subspace $\fH'_j$ of
$L$ associated with the spectral subset $\sigma'_j=\spec(L)\cap
O_{r_V}(\sigma_j)$, $j=0,1$. The estimate \eqref{Bosc} improves the
corresponding bound for the one-dimensional case obtained in
\cite[Section 6]{AlMoSh}  (cf. Re\-mark \ref{Rgimp}).

\smallskip

\vspace*{2mm} \noindent {\bf Acknowledgements.} A.\,K.~Motovilov
gratefully acknowledges the kind hospitality and support of the
Mathe\-matisches Institut, Universit\"at Bern; his research was also
supported by Deutsche For\-sch\-ungs\-gemeinschaft (DFG), Grant no.\
436 RUS 113/817, by Russian Foundation for Basic Research, Grants
no. 06-01-04003 and 09-01-90408, and by the Heisenberg-Landau
Program. C.~Tretter kindly acknowledges the support of this work by
Deutsche Forschungsgemeinschaft (DFG), Grant no.\ TR~368/6-2, and by
Schweizerischer Nationalfonds (SNF), Grant no.\ 200021-119826/1. The
authors are grateful to H.\ Langer for drawing their attention to
the papers \cite{Veselic1,Veselic2} by K.\ Ve\-se\-li\'c;
A.\,K.~Motovilov also thanks K.\ Ve\-se\-li\'c for enlightening
discussions.


\end{document}